\documentclass[hidelinks,onefignum,onetabnum]{siamart220329}

\usepackage{amsmath}
\usepackage{amssymb}
\usepackage{amsfonts}
\usepackage{enumerate}
\usepackage{graphicx}
\usepackage{xcolor}
\usepackage[sort,compress]{cite} 
\usepackage{bm}
\usepackage{cleveref}

\def\RR{{\mathbb R}}
\def\QQ{{\mathbb Q}}

\def\NN{\mathbb{N}}

\def\zero{\bm{0}}
\def\sign{\,\mathrm{sign}}

\def\RRge{\RR_{\ge0}}
\def\RRgt{\RR_{>0}}

\def\I{\mathcal{I}}
\def\J{\mathcal{J}}

\newcommand{\abs}[1]{\left|#1\right|}
\newcommand{\norm}[2][]{\left\lVert #2 \right\rVert_{#1}}

\newcommand{\diffd}{\mathrm{d}}

\newcommand{\sbrm}[1]{\sb{\mathrm{#1}}}

\newcommand{\TT}{^{\mathrm{T}}}

\newcommand{\g}{\bm{g}}

\newcommand{\R}{\mathbf{R}}
\newcommand{\F}{\mathcal{F}}
\renewcommand{\S}{\mathcal{S}}
\newcommand{\FL}{\F_L}
\newcommand{\FLR}{\F_{L,\R}}
\newcommand{\E}{\mathcal{E}}
\newcommand{\EN}{\E_N}
\newcommand{\meas}{u}
\newcommand{\diffout}{\bm{\mathrm{y}}}

\newcommand{\Diff}{\mathcal{D}}

\newcommand{\M}{\bm{\mathrm{M}}}
\newcommand{\Q}{\bm{\mathrm{Q}}}

\newcommand{\gop}{\bar{g}}
\newcommand{\goptT}[2]{\gop_{N,#2}^{m,L}(#1)}
\newcommand{\gopt}[1]{\goptT{T}{#1}}

\newtheorem{example}{Example}

 \newsiamremark{remark}{Remark}
\newsiamremark{clm}{Claim}
\newsiamthm{defin}{Definition}
\newtheorem{eexample}{Example}

\headers{DETERMINISTICALLY OPTIMAL DIFFERENTIATORS}{R. SEEBER AND H. HAIMOVICH}

\title{Deterministically Optimal\\ Robust Exact Differentiators \\ of arbitrary order\thanks{\funding{Agencia I+D+i grant PICT-2021-I-A-00730, Argentina.}} }

\author{Richard Seeber\thanks{Graz University of Technology, Graz, Austria (\email{richard.seeber@tugraz.at})}
  \and Hernan Haimovich\thanks{Centro Internacional Franco-Argentino de Ciencias de la Informaci\'on y Sistemas (CIFASIS), UNR-CONICET, Rosario, Argentina
  (\email{haimovich@cifasis-conicet.gov.ar}).}}

\begin{document}

\maketitle

\begin{abstract}
Estimation of the derivatives of a function with bounded high-order derivative in the presence of bounded measurement noise is considered, in a deterministic setting. 
Theoretical fundamental limitations of causal differentiators in terms of the lowest achievable worst-case differentiation error and desired properties such as exactness---the zero-error estimation of derivatives in the absence of noise---and robustness---the small sensitivity of the estimate under small perturbations---are established for the first time for arbitrary differentiation order.
Differentiators that achieve these theoretically lowest bounds on their differentiation error are formally defined, fully characterized, and their properties are studied.
In particular, deterministically optimal differentiators---those featuring optimal differentiation error bounds among the class of exact differentiators---are shown to exist by means of a novel construction exhibiting, in addition to exactness from the beginning, a very strong form of robustness.
\end{abstract}



\section{Introduction}
\label{sec:intro}

The estimation of derivatives of a signal based on noisy or corrupted measurements is an old problem that has many practical applications. This problem, named signal differentiation, can be approached in different ways, depending on the information about the signal to be differentiated and the noise or perturbations affecting the measurements. Approaches to signal differentiation can be classified broadly as stochastic or deterministic. The stochastic approach involves some sort of statistical description for the noise and/or signal.
The deterministic one, in contrast, does not, but can involve hard bounds on specific variables, for example.

Another broad classification of differentiators is that of being causal or noncausal, a differentiator being the method or algorithm that produces the required estimates for the signal derivatives. Causal differentiators can produce an estimate based solely on past and present measurements, whereas noncausal ones can depend on future ones and typically are given blocks of measurements to estimate the derivatives that the signal had at any time during which the measurements were taken. Causal differentiators are suitable for online or real-time differentiation, as required for automatic control, whereas noncausal ones are suitable for post-processing blocks of data in order to estimate the whole functional form of the signal's derivatives, as required in plasma physics \cite{druyve_zphys30}, for example.

This paper studies causal differentiators from a deterministic approach. Methods for causal signal differentiation that can be suitable in a deterministic setting include algebraic methods \cite{ibrdio_ijamcs04,mbojoi_na09}, linear high-gain observers \cite{vasiljevic2008error,khapra_ijrnc14} and sliding-mode differentiators \cite{yuxu_el96,levant1998robust,levant_ijc03}. 
The different signal differentiation methods have varying degrees of performance depending on certain features of the signal and noise. 

Knowledge of a bound for some higher-order derivative of the signal to be differentiated is a usual assumption in relation to sliding-mode control \cite{levant1998robust,levant_ijc03}. This knowledge can be employed to design differentiators that output the exact value of the input signal's lower-order derivatives in the absence of noise after a finite-time transient, a property called exactness, and in addition so that noise of small amplitude perturbs the exact output only slightly, a property called robustness. Linear differentiators are inherently robust but cannot be exact except on a very narrow signal class, namely polynomials with a known bound on the degree. Sliding-mode differentiators, in contrast, can be exact on larger signal classes but this exactness also somewhat deteriorates their worst-case behavior in presence of noise \cite{levant2017sliding,seehai_auto23}.
The latter fact points at a feature that any differentiator that is exact on a large signal class exhibits: if the measurement noise is differentiable and of small amplitude, it can be theoretically impossible to distinguish the measured signal from a clean signal without noise.
The specific limitation in that regard is related to Landau-Kolmogorov type inequalities, relating bounds on a function to those of its derivatives \cite{levant2017sliding}.
These types of problems have been studied extensively in literature, see e.g. \cite{schcav_tr70,bagdas_book98,karlin_inbook76}.

Generalizations of the robust exact differentiators of \cite{levant1998robust} abound, especially for differentiation order one. For instance, \cite{cruzzavala2011uniform} develops a first-order robust and exact differentiator with fixed-time convergence, i.e.~where the finite convergence time is bounded irrespective of the initial conditions, and \cite{seehai_auto21} provides a design framework for differentiators with a user-defined, uniform bound on the convergence time. As regards arbitrary differentiation orders, \cite{levliv_tac12} relaxes the assumption of a constant bound on the high-order derivative to allow for some known time-varying function, whereas \cite{angmor_auto13} develops arbitrary-order differentiators with fixed-time convergence by leveraging homogeneity. Differentiators having the latter features were precisely analyzed in \cite{moreno_tac22}, for which also a design strategy is provided. \cite{levyu_tac18} develop exact differentiators that are not only robust to noise of small amplitude but also to possibly unbounded noise with small integral, and by replacing the assumption of bounded noise by a bounded (repeated) integral of the noise, \cite{levliv_ejc20} develop the so-called filtering differentiators. Despite the variety of existing generalizations and analyses, the fundamental theoretical limitations of robust and exact differentiators in the presence of bounded measurement noise, under knowledge of a constant bound for some high-order derivative, and in relation to their worst-case differentiation error and corresponding convergence time remains mostly unexplored, especially for differentiation orders greater than one, with the exception of \cite{seehai_auto23,seehai_cdc25,aldsee_tac25}.

In \cite{seehai_auto23}, a first-order causal differentiator that is exact from the beginning and robust almost from the beginning, i.e. that exhibits the fastest possible convergence speed among causal, exact and robust differentiators, was constructed for the first time. That construction, although based on the online adaptation of the time-length parameter of a standard difference quotient, involves many ad hoc steps with no obvious generalization to higher orders of differentiation. In \cite{seehai_cdc25}, the construction of \cite{seehai_auto23} is generalized to second-order differentiation by proposing a new interpretation of the construction. Unfortunately, even that generalization seems to be inappropriate for differentiation orders greater than two, and indeed, it is left unclear whether such generalization to arbitrary orders is in fact possible. 

In this paper, theoretical questions related to the best performance features of causal differentiators are posed and answered for the first time for arbitrary differentiation order. Specifically, precise definitions of optimal differentiators are developed in a deterministic setting in relation to the lowest achievable worst-case error \emph{at every time instant}, and the existence of these optimal differentiators is demonstrated constructively. These deterministically optimal differentiators are shown to be exact from the beginning, i.e. faster than existing fixed-time convergent differentiators (with the exception of \cite{seehai_auto23}), and to have a very strong form of robustness almost from the beginning, whereby small perturbations of \emph{any differentiator input} yield only small deviations of the corresponding output. To the best of the authors' knowledge, the explicit consideration of this form of strong robustness is novel and provided for the first time in the current paper. The exactness and robustness properties are considered in relation to the convergence time, as opposed to the standard asymptotic results of Levant in \cite{levant1998robust}, and they generalize the notions first proposed in \cite{seehai_auto23} to arbitrary differentiation order.

The remainder of this paper is organized as follows. A brief explanation of the notation employed is provided at the end of the current section. Section~\ref{sec:problem} introduces the specific problem addressed and the essential definitions and function spaces related to differentiators. The standard causality property as well as the specific exactness and robustness ones are introduced in Section~\ref{sec:features}, and a fundamental limitation regarding the latter two is shown.
Tight lower bounds on the worst-case differentiation error for general causal as well as for causal and exact differentiators are derived in Section~\ref{sec:worst-case-error}. These bounds are employed to precisely define deterministically optimal differentiators in Section~\ref{sec:opt-diff}, where the properties of these are also extensively studied, and their existence is established constructively.
Conclusions are drawn in Section~\ref{sec:conclusions}. 

\textbf{Notation:}
$\RR$, $\RR_{\ge 0}$, $\RR_{>0}$ denote the reals, nonnegative reals, and positive reals, respectively. $\NN$, $\NN_0$ and $\QQ$ denote the natural numbers, the naturals including 0, and the rational numbers, respectively. 
If $\alpha\in\RR$, then $|\alpha|$ denotes its absolute value.
For a set $A$ and a function $f$, the image of $A$ under $f$ is denoted by $f(A) = \{ f(a) : a \in A \}$.
The $i$-th derivative $\frac{\diffd^i f}{\diffd t^i}$ of a function $f$ of one variable is commonly written as $f^{(i)}$, with $\dot f = f^{(1)}$ and $\ddot f = f^{(2)}$ denoting its first and second derivative in particular.
One-sided limits of a function $f$ at time instant $T$ from above or below are written as $\lim_{t \to T^+} f(t)$ or $\lim_{t \to T^-} f(t)$, respectively. `Almost everywhere' is abbreviated as `a.e.'. If $x,y \in \RR^n$, then $x\ge y$ denotes the componentwise inequalities $x_i \ge y_i$ for $i\in \{1,\ldots,n\}$.
For $x\in\RR^n$, $\norm[\infty]{x}:= \max_{i=1,\ldots,n} |x_i|$ denotes its infinity norm. 
The bold symbols $\zero$, $\mathbf{1}$ denote vectors all of whose components equal 0 and 1, respectively, whose dimensions are given by the context.

\section{Preliminaries and Problem Statement}
\label{sec:problem}

This section introduces the considered problem of signal differentiation from noisy measurements and defines the considered classes of signals, noise and differentiator inputs.

Let $m\ge 1$ and consider the estimation of the first $m$ derivatives of a function $f : \RR_{\ge 0} \to \RR$ based on noisy measurements $\meas = f + \eta$, under the assumption that uniform bounds $N$ and $L$ for the noise $\eta$ and the $m+1$-th derivative $f^{(m+1)}$, respectively, exist.
More precisely, let $\F^m$ denote the set of functions $f : \RR_{\ge 0} \to \RR$ such that $f$ is $m$ times continuously differentiable and $f^{(m)}$ is Lipschitz continuous on $\RR_{\ge 0}$. The fact that $f\in\F^m$ then implies that the $m+1$-th derivative $f^{(m+1)}$ exists almost everywhere (a.e.) due to Rademacher's Theorem. The corresponding classes of signals to consider, from which the measurements are generated, are hence given by
\begin{subequations}
    \begin{align}
        \FL^m &= \{ f \in \F^m : \abs{f^{(m+1)}(t)} \le L \text{ a.e. on } \RR_{\ge 0}\} \\
        \EN &= \{ \eta : \RR_{\ge 0} \to \RR  : \abs{\eta(t)} \le N \text{ for all }t\ge 0 \}.
    \end{align}
\end{subequations}
Write $\FL^m + \EN = \{ f + \eta : f \in \FL^m, \eta \in \EN\}$ for the set of inputs $u$ with fixed $L$, $N$.
All conceivable inputs to a differentiator then belong to the set
\begin{equation}
    \mathcal U = \bigcup_{\substack{m \in \NN_0 \\ L,N \in \RR_{\ge 0}}} (\FL^m + \EN).
\end{equation}
It is clear that all functions in $\mathcal U$ are locally bounded, and that $\mathcal{U}$ contains, in particular, all uniformly bounded functions and all polynomials.
In that regard, note also that $\F_0^m$ is the set of polynomials of degree not greater than $m$. Both $\mathcal{U}$ and $\F_0^m$ are vector spaces over $\RR$ with the standard operations of function addition and scalar multiplication.
Whenever $m \in \NN$, $i\in\NN_0$, $i\le m$ and $f\in\FL^m$, the notation $f^{(i)}(0)$ will be used to denote the one-sided limit $\lim_{t\to 0^+} f^{(i)}(t)$, whose existence is ensured, and $f^{(0)} := f$.

A differentiator of order $\ell$ is an operator $\Diff : \mathcal U \to ( \RR_{\ge 0} \to \RR^\ell)$ that maps the measured signal $\meas$ to an estimate $\diffout = \Diff \meas$ of the vector $[\dot f, \ddot f, \cdots, f^{(\ell)}]\TT$ of the first $\ell$ derivatives of $f$.
The estimates for individual derivatives will be indicated by a subscript, so that $y_i = \Diff_i \meas$ denotes the estimate of the $i$-th derivative.

\section{Features and Properties of Differentiators}
\label{sec:features}

This section introduces important basic properties of differentiators.
We start with the standard definition of causality.

\begin{defin}[Causality]
A differentiator $\Diff$ is said to be causal if, for all $t \in \RR_{\ge 0}$, $[\Diff u_1](t) = [\Diff u_2](t)$ holds whenever $u_1(\tau) = u_2(\tau)$ holds for all $\tau \in [0, t]$.
\end{defin}

\subsection{Exactness and Robustness}

Two further important properties are exactness and robustness, as introduced by \cite{levant1998robust}. 
To formally define these properties, the worst-case differentiation error is introduced next.
\begin{defin}[Worst-case differentiation error]
    \label{def:acc:abs}
    Let $L, N \in \RR_{\ge 0}$, let $m,i\in\NN$, $i\le m$.
    A differentiator $\Diff$ of order $m$ is said to have
worst-case error $M_{N,i}^{\mathcal{S}}(t)$ for the $i$-th derivative from time $t\ge 0$ over the signal class $\mathcal{S} \subseteq \FL^m$ with noise bound $N$ if
            \begin{equation}
                M_{N,i}^{\mathcal{S}}(t) = \sup_{\substack{u=f+\eta \\ \eta\in\EN\\ f\in\mathcal{S}}} \sup_{\tau \ge t} \big|f^{(i)}(\tau) - [\Diff_i u](\tau)\big|
            \end{equation}
\end{defin} 
Let $\M_N^{\S}$ denote the column vector whose components are $M_{N,i}^{\S}$ for $i=1,\ldots,m$.

In practice, a differentiator may well have infinite worst-case error over the signal class $\FL^m$ at every finite time instant $t$, because $\FL^m$ contains signals with arbitrarily large initial conditions\footnote{``Initial conditions'' of $f\in\FL^m$ refers to the values $f^{(i)}(0)$ for all $i\in\NN_0$ with $i\le m$.}.
To still analyze worst-case behavior in such cases, define for all $R_0, \ldots, R_{m}\ge 0$ the class of signals with bounded $m+1$-th derivative that, in addition, have bounded initial values and lower-order derivatives, as 
\begin{equation}
    \FLR^m := \{ f \in \FL^m : |f^{(i)}(0)| \le R_i, i=0,\ldots,m \},
\end{equation}
where $f^{(0)} := f$ and $\R = [R_0 \quad R_1 \quad \ldots \quad R_{m}]\TT \in \RR_{\ge 0}^{m+1}$.
With this definition, the worst-case error $M_{N,i}^{\FLR^m}(t)$ over $\FLR^m$ is nonincreasing with respect to $t$ and nondecreasing with respect to $N$, $L$ and every component of $\R$.
Moreover, the worst-case error over $\FL^m$ is given by $M_{N,i}^{\FL^m}(t) = \sup_{\R \in \RR_{\ge 0}^{m+1}} M_{N,i}^{\FLR^m}(t)$.

The different notions of exactness and robustness are next recalled from \cite{seehai_auto23} and generalized to arbitrary differentiation order.
\begin{defin}[Exactness]
\label{def:exact}
    Let $L \in \RRge$, $m\in\NN$.
    A differentiator $\Diff$ of order $m$ is said to be 
    \begin{itemize}
        \item
            exact from the beginning over $\FL^m$ if $\M_0^{\FL^m}(t) = \zero$ for all $t > 0$;
        \item
            exact in fixed time $T \ge 0$ over $\FL^m$ if $\M_0^{\FL^m}(T) = \zero$;
        \item
            exact in finite time over $\FL^m$ if for every $\R \in \RR_{\ge 0}^{m+1}$ there exists a $t_{\R} > 0$ such that $\M_0^{\FLR^m}(t_{\R}) = \zero$;
        \item
            not exact over $\FL^m$ if it is not exact in finite time over $\FL^m$.
    \end{itemize}
    A differentiator $\Diff$ is said to be exact over polynomials of degree at most $m$ whenever it is exact over $\F_0^m$.
\end{defin}
\begin{defin}[Worst-case sensitivity]
Let $\mathcal{S} \subseteq \mathcal{U}$, $\epsilon \in \RRgt$, $m,i\in\NN$, $i \le m$. A differentiatior $\Diff$ is said to have worst-case sensitivity $Q^{\mathcal{S}}_{\epsilon,i}(t)$ for the $i$-th derivative from time $t\ge 0$ over the signal class $\mathcal{S}$ under perturbations bounded by $\epsilon$ if
\begin{align}
\label{eq:def:Q}
    Q^{\mathcal{S}}_{\epsilon,i}(t) &= \sup_{\substack{u_2=u_1+\eta \\ \eta\in\E_{\epsilon}\\ u_1 \in \mathcal{S}}} \sup_{\tau \ge t} \big| [\Diff_i u_1](\tau) - [\Diff_i u_2](\tau)\big|,\\
    \intertext{and $Q^{\mathcal{S}}_i(t)$ under infinitesimally small perturbations if}
    Q^{\mathcal{S}}_i(t) &= \lim_{\epsilon \to 0^+} Q^{\mathcal{S}}_{\epsilon,i}(t).
\end{align}
\end{defin}

Note that the limit in \eqref{eq:def:Q} always exists if $Q_{\epsilon,i}^{\mathcal{S}}(t)$ is finite for some $\epsilon$, because the latter is non-decreasing with respect to $\epsilon$ and is bounded from below by zero.
The quantity $Q^{\mathcal{S}}_{\epsilon,i}(t)$ equals the maximum deviation from time $t$ of the $i$-th differentiator output, for signals in $\mathcal{S}$ under perturbations bounded by $\epsilon$.
Let $\Q^{\mathcal{S}}_\epsilon(t)$ and $\Q^{\mathcal{S}}(t)$ denote the column vectors whose components are $Q^{\mathcal{S}}_{\epsilon,i}(t)$ and $Q^{\mathcal{S}}_i(t)$, respectively, for $i=1,\ldots,m$.

\begin{defin}[Robustness]
\label{def:robust}
    A differentiator $\Diff$ of order $m\in\NN$ is said to be
    \begin{itemize}
\item robust from the beginning over $\mathcal{S} \subseteq \mathcal{U}$, if $\Q^{\mathcal{S}}(0) = \zero$;
        \item robust almost from the beginning over $\mathcal{S} \subseteq \mathcal{U}$, if
 $\Q^{\mathcal{S}}(t) = \zero$ for all $t > 0$.
    \end{itemize}
\end{defin}
When the value of $L$ and the differentiation order $m$ are clear from the context, a differentiator that is robust over $\FL^m$ will just be called \emph{robust} in the following, while a differentiator that is robust over the largest set $\mathcal{S} \subseteq \mathcal{U}$ on which it is defined\footnote{For differentiators defined using differential equations or inclusions, this set $\mathcal{S}$ is usually the subset of all Lebesgue measurable functions in $\mathcal{U}$; for the differentiators introduced in this paper, this set is $\mathcal{U}$.} will be called \emph{strongly robust}.
Note that robustness, in this sense, considers only infinitesimal perturbations of noise-free input signals in $\FL^m$.
Strong robustness, on the other hand, considers infinitesimal perturbations of any input signal for which the differentiator is defined, and is thus indeed the strongest robustness property that may conceivably be defined.

\subsection{Limitations of Exactness and Robustness}

The following proposition shows that there are no differentiators that are both exact and robust from the beginning; hence, the best combination of properties one can hope to achieve, in some sense, is exactness from the beginning and robustness almost from the beginning.
\begin{proposition}
    \label{th:connections:causal-exact}
    Let $m \in \NN$, $L \in \RR_{> 0}$ and consider a causal differentiator $\Diff$ of order $m$ that is exact from the beginning over $\F_{0}^m$.
    Then, $\Diff$ is not robust from the beginning over $\E_0$.
\end{proposition}
\begin{remark}
Note that $\E_0$ contains only the zero function.
It is thus the smallest set over which robustness may be considered and differentiators lacking robustness over $\E_0$ lack robustness also over any other considered signal set.
\end{remark}
\begin{proof}
    Note that $\E_0 = \F_{0,\bm{0}}^m$.
    It will be shown that $Q^{\F_{0,\bm{0}}^m}_{N,k}(0) = \infty$ for all $N \in \RR_{> 0}$.
    To obtain a contradiction, assume $M_{0,k}^{\F_{0}^m}(\tau) = 0$ for all $\tau > 0$ and suppose existence of $c \in \RR_{> 0}$ such that  $Q^{\F_{0,\bm{0}}^m}_{N,k}(0) < c$ holds.
    Consider $f \in \F_{0,\bm{0}}^m$, $g \in \F_{0}^m$ defined as $f(t) = 0$, $g(t) = c t^k$, and define $u(t) = \min\{ g(t), N \}$.
    Let $\tau = \sqrt[k]{c^{-1} N}$ and note that $u(t) = g(t)$ for all $t \in [0, \tau]$.
    Then, due to causality and $M_{0,k}^{\F_0^k}(\tau) = 0$, it satisfies $[\Diff_k u](\tau) = [\Diff_k g](\tau) = g^{(k)}(\tau) = ck!$ as well as $[\Diff_k f](\tau) = f^{(k)}(\tau) = 0$.
    Then, $|[\Diff_k f](\tau) - [\Diff_k u](\tau)| = g^{(k)}(\tau) \ge c$, yielding the contradiction $Q^{\F_{0,\bm{0}}^m}_{N,k}(0) \ge c$ since $f \in \F_{0,\bm{0}}^m$ and $f-u \in \EN$. 
\end{proof}

\section{Lower Bounds for the Worst-case Differentiation Error}
\label{sec:worst-case-error}

In order to characterize lower bounds for the worst-case differentiation error of arbitrary causal or causal and exact differentiators, the worst-case differentiation error corresponding to specific bounded signals will first be analyzed.

\subsection{Worst-case Error for Bounded Functions}
For all $L,N\in\RRge$, $m\in\NN$, and $T\in\RRge$, define the set
\begin{align}
\label{eq:def:G}
    \mathcal{G}_{L,N}^{m}(T) &= \big\{ g\in\FL^{m} : |g(t)| \le N \text{ for all } t \in [0,T] \big\}
\end{align}
and consider, for $T>0$, the following Landau-Kolmogorov problem on the finite interval $[0,T]$
\begin{align}
\label{eq:def:gopt}
    \gopt{k} &= \sup \{ g^{(k)}(T) : g \in \mathcal{G}_{L,N}^{m}(T) \}.
\end{align}
Additionally, for $T\ge 0$, define
\begin{equation}
    \label{eq:def:gopt0}
        \underline g_{N,k}^{m,L}(T) = \sup \{ g^{(k)}(T) : g \in \mathcal{G}_{L,N}^{m}(T) \cap \F_{L,\bm{0}}^m \}.
\end{equation}
It is clear that both $\bar g_{N,k}^{m,L}(T)$ and $\underline g_{N,k}^{m,L}(T)$ are non-decreasing with respect to $N, L$.
Moreover, the former satisfies $\bar g_{0,k}^{m,L}(T) = 0$, $\lim_{T \to \infty} \bar g_{N,k}^{m,0}(T) = 0$ for all $L,N \in \RR_{\ge 0}$, $T \in \RR_{> 0}$ and the latter fulfills $\underline g_{0,k}^{m,L}(T) = \underline g_{N,k}^{m,0}(T) = \underline g_{N,k}^{m,L}(0) = 0$ for all $L,N,T \in \RR_{\ge 0}$.
The following lemma establishes some further important properties.
\begin{lemma}
\label{lem:gopt-limit}
    Let $L, N \in \RR_{\ge 0}$, $k,m\in\NN$, $k\le m$.
    Then,
    \begin{enumerate}[a)]
        \item
        \label{it:gopt-limit:ineq}
        $\underline g_{N,k}^{m,L}(T) \le \bar g_{N,k}^{m,L}(T)$ holds for all $T \in \RR_{> 0}$;
        \item
        \label{it:gopt-limit:nonincreasing}
        $\bar g_{N,k}^{m,L}(T)$ is non-increasing with respect to $T$;
        \item
        \label{it:gopt-limit:nondecreasing}
        $\underline g_{N,k}^{m,L}(T)$ is non-decreasing with respect to $T$;
        \item 
        \label{it:gopt-limit:limits}
        and $
        \lim_{\tau \to \infty} \underline g_{N,k}^{m,L}(\tau) = \lim_{\tau \to \infty} \bar g_{N,k}^{m,L}(\tau).
    $
    \end{enumerate}
\end{lemma}
\begin{proof}
    Item~\ref{it:gopt-limit:ineq}) is clear from the additional constraint that is present in \eqref{eq:def:gopt0} compared to \eqref{eq:def:gopt}.
    To see item~\ref{it:gopt-limit:nonincreasing}), note that, due to symmetry with respect to time reversal, $\gopt{k}$ may equivalently be written as
    \begin{equation}
    \label{eq:gopt:equiv}
        \gopt{k} = \sup \{ f^{(k)}(0) : f \in \mathcal{G}_{L,N}^{m}(T) \}
    \end{equation}
    and that $T_2 \ge T_1$ implies $\mathcal{G}_{L,N}^{m}(T_2) \subseteq  \mathcal{G}_{L,N}^{m}(T_1)$.
    To see item~\ref{it:gopt-limit:nondecreasing}), note that for every $g \in \mathcal{G}_{L,N}^{m}(T) \cap \F_{L,\bm{0}}^m$ and every $\theta \ge 0$, the function $\tilde g$ satisfying $\tilde g(\tau) = 0$ for all $\tau \in [0, \theta)$ and $\tilde g(\tau) = g(\tau-\theta)$ for $\tau \in [\theta, T+\theta]$ satisfies $\tilde g \in \mathcal{G}_{L,N}^{m}(T+\theta) \cap \F_{L,\bm{0}}^m$ and $\tilde g^{(k)}(T+\theta) = g^{(k)}(T)$.

    Since $\underline g_{N,k}^{m,L}$ is non-decreasing, $\bar g_{N,k}^{m,L}$ is non-increasing, and $\underline g_{N,k}^{m,L}(T) \le \bar g_{N,k}^{m,L}(T)$, both have well-defined limits.
    It remains to show $\lim_{\tau \to \infty} \underline g_{N,k}^{m,L}(\tau) \ge \lim_{\tau \to \infty} \bar g_{N,k}^{m,L}(\tau)$ to prove item~\ref{it:gopt-limit:limits}).
    For $L = 0$, this follows from $\lim_{\tau \to \infty} \bar g_{N,k}^{m,0}(\tau) = 0$.
    To show it also for $L > 0$, consider any $T > 0$ and $\R \in \RR_{\ge 0}^{m+1}$ sufficiently large such that $\mathcal{G}_{L,N}^{m}(T) \subseteq \FLR^m$, and hence $\mathcal{G}_{L,N}^{m}(t) \subseteq \FLR^m$ for all $t \ge T$.
    Then, there exists a sufficiently small $\lambda \in (0,1)$ and some $\theta \in \RR_{> 0}$ such that for each $g \in \F_{L,\lambda \R}^m$ there exists $h_g \in \F_{L,\bm{0}}^m$ with $|h_g(\tau)| \le N$ for $\tau \in [0,\theta]$ and $h_g(\tau) = g(\tau-\theta)$ for $\tau \ge \theta$.
    Indeed, after fixing $\theta$, such $\lambda$ and $h_g$ may be obtained, for example, by using a Hermite interpolation polynomial.
For each $t \ge T$ and $g_1 \in \mathcal{G}_{L,N}^{m}(t)$, $g_2 \in \mathcal{G}_{L,N}^{m}(t) \cap \F_{L,\bm{0}}^m$, now define the function $g = \lambda g_1 + (1-\lambda) g_2\in \F_{L,\lambda \R}^m$.
    Then, $h_g$ constructed as above fulfills $h_g \in \mathcal{G}_{L,N}^{m}(t+\theta) \cap \F_{L,\bm{0}}^m$.
    Consequently, $\underline g_{N,k}^{m,L}(t+\theta) \ge h_g^{(k)}(t+\theta) = \lambda g_1^{(k)}(t) + (1-\lambda) g_2^{(k)}(t)$ and
    taking the supremum over all $g_1 \in \mathcal{G}_{L,N}^{m}(t), g_2 \in \mathcal{G}_{L,N}^{m}(t) \cap \F_{L,\bm{0}}^m$ yields
    the inequality $
        \underline g_{N,k}^{m,L}(t+\theta) -  (1-\lambda) \underline g_{N,k}^{m,L}(t) \ge \lambda \bar g_{N,k}^{m,L}(t)
    $
    for all $t \ge T$.
    Taking limits on both sides  and dividing by $\lambda$ yields the claim.
\end{proof}

The next lemma establishes some important scaling properties of $\gopt{k}$.
\begin{lemma}
\label{lem:gopt-scaling}
    Let $L,N \in \RR_{\ge 0}$, $M \in [0,N]$, $k,m \in \NN$, $k \le m$.
    Then, the relations
    \begin{align}
    \label{eq:gopt-scaling}
        \bar g_{\beta N,k}^{m,\alpha L}(T) &= \sqrt[m+1]{\alpha^k \beta^{m+1-k}} \, \bar g_{N,k}^{m,L}(\sqrt[m+1]{\alpha \beta^{-1}} T) & & \text{for all $\alpha, \beta \in \RR_{>0}$}, \\
        \label{eq:gopt-homogeneity}
        \bar g_{\gamma N,k}^{m,\gamma L}(T) &= \gamma \bar g_{N,k}^{m,L}(T) & & \text{for all $\gamma \in \RR_{\ge 0}$,} \\
        \label{eq:gopt-timescaling}
        \bar g_{N,k}^{m,L}(T) &\ge \gamma^k \bar g_{N,k}^{m,L}(\gamma T) & & \text{for all $\gamma \in (0,1]$, and} \\
        \label{eq:gopt-addnoise}
        \bar g_{N+M,k}^{m,L}(T) &\le \bar g_{N,k}^{m,L}(T) + \bar g_{M,k}^{m,L/2}(T)
    \end{align}
    hold for each $T \in \RR_{> 0}$.
\end{lemma}
\begin{proof}
    Let $n = m+1$ for simplicity.
    As for \eqref{eq:gopt-scaling}, use the substitution $f(\tau) = \beta g(\sqrt[n]{\alpha \beta^{-1}} \tau)$ in the equivalent representation \eqref{eq:gopt:equiv} of $\gopt{k}$ to obtain
    \begin{align*}
        \gop^{m,\alpha L}_{\beta N,k}(T) &= \sup\big\{ f^{(k)}(0) :
    f\in\F_{\alpha L}^m, |f(\tau)| \le \beta N \text{ for all } \tau \in [0,T] \big\}  \\
    &= \sup\big\{ \sqrt[n]{\alpha^k \beta^{n-k}} g^{(k)}(0) : g \in \FL^m, |g(\sqrt[n]{\alpha \beta^{-1}} \tau)| \le N \text{ for all }  \tau \in [0, T] \big\}   \\
    &= \sqrt[n]{\alpha^k \beta^{n-k}} \sup\big\{ g^{(k)}(0) : g \in \FL^m, |g(t)| \le N \text{ for all } t \in [0, \sqrt[n]{\alpha \beta^{-1}} T] \big \}.
    \end{align*}
    To prove \cref{eq:gopt-homogeneity}, note that it is trivial for $\gamma = 0$; otherwise, it is obtained by setting $\alpha = \beta = \gamma$ in \eqref{eq:gopt-scaling}.
    To show \cref{eq:gopt-timescaling}, use monotonicity with respect to $L$ and \eqref{eq:gopt-scaling} to obtain $\gop_{N,k}^{m,L}(T) \ge \gop_{N,k}^{m,\gamma^n L}(T) = \gamma^k \gop_{N,k}^{m,L}(\gamma T)$.
    For \cref{eq:gopt-addnoise}, which is trivial if $N=M=0$, define otherwise $\gamma_1 = N/(N+M)$, $\gamma_2 = M/(N+M) \le \frac{1}{2}$ and obtain
    \begin{align*}
        \gop^{m,L}_{N+M,k}(T) &= (\gamma_1 + \gamma_2) \gop^{m,L}_{N+M,k}(T) 
        = \gop^{m,\gamma_1 L}_{N,k}(T) + \gop^{m,\gamma_2 L}_{M,k}(T) 
        \le \gop^{m,L}_{N,k}(T) + \gop^{m,L/2}_{M,k}(T),
    \end{align*}
    where the second equality follows after distributing the product over the sum and employing \eqref{eq:gopt-homogeneity}, and the inequality from monotonicity with respect to $L$.
\end{proof}

The next lemma shows the limit of $\gopt{k}$ to be strictly concave in $N$ for $L > 0$.
\begin{lemma}
\label{lem:gopt:concavelimit}
Let $L,N \in \RR_{> 0}$, $M \in \RR_{\ge 0}$, $m \in \NN$ and $k \in \{1, \ldots, m\}$.
Then,
\begin{equation}
    \label{eq:gopt-concave}
        \lim_{T \to \infty} \bar g_{\lambda N + (1-\lambda)M,k}^{m,L}(T) > \lambda \lim_{T \to \infty} \bar g_{N,k}^{m,L}(T) + (1-\lambda) \lim_{T \to \infty} \bar g_{M,k}^{m,L}(T)
        \end{equation}
        holds for all $\lambda \in (0,1)$.
\end{lemma}
\begin{proof}
Let $n=m+1$. Use the scaling relation \eqref{eq:gopt-scaling} to obtain
    \begin{equation}
        \label{eq:lim:barg:scaled}
        \lim_{T \to \infty} \bar g_{\beta N,k}^{m,L}(T) = \lim_{T \to \infty} \beta^\frac{n-k}{n} \, \bar g_{N,k}^{m,L}(\beta^{-\frac{1}{n}} T) = \beta^\frac{n-k}{n} \, \lim_{T \to \infty} \bar g_{N,k}^{m,L}(T)
    \end{equation}
    for all $\beta > 0$, wherein $\lim_{T \to \infty} \bar g_{N,k}^{m,L}(T) \ge \underline g_{N,k}^{m,L}(1) > 0$ since $L,N > 0$.
    Setting $\beta = \lambda + (1-\lambda) \gamma$ with $\gamma = \frac{M}{N} \in \RR_{\ge 0}$ then yields
    \begin{align}
         \lim_{T \to \infty} \bar g_{\beta N,k}^{m,L}(T) &= \left[\lambda + (1-\lambda)\gamma\right]^{\frac{n-k}{n}} \, \lim_{T \to \infty} \bar g_{N,k}^{m,L}(T) \nonumber
         > \left[\lambda + (1-\lambda)\gamma^{\frac{n-k}{n}}\right] \lim_{T \to \infty} \bar g_{N,k}^{m,L}(T) \nonumber\\
         \label{eq:lim:barg:concave}
         &= \lambda \lim_{T \to \infty} \bar g_{N,k}^{m,L}(T) + (1-\lambda) \lim_{T \to \infty} \bar g_{\gamma N,k}^{m,L}(T),
    \end{align}
    where the inequality follows from concavity of the function $x \mapsto x^p$ when $0<p=\frac{n-k}{n}<1$ and the fact that $\lambda,1-\lambda \in (0,1)$, and the last equality from \eqref{eq:lim:barg:scaled}. Since $\beta N = \lambda N + (1-\lambda) M$ and $\gamma N = M$, the inequality \eqref{eq:gopt-concave} follows from \eqref{eq:lim:barg:concave} and the proof is complete.
\end{proof}

The next two lemmata show that the the supremum in \eqref{eq:def:gopt} may also be replaced by the maximum, i.e., there always exists $g \in \mathcal{G}_{L,N}^m(T)$ with $g^{(k)}(T) = \gopt{k}$.
Moreover, it is shown that such $g$ can always be chosen such that it satisfies $|g(\tau)| = N$ only at finitely many, hence isolated, time instants.
\begin{lemma}
    \label{lem:G:der:bounded}
    Let $L, N \in \RR_{\ge 0}$, $T \in \RR_{> 0}$, and consider a sequence of functions $h_s \in \mathcal{G}_{L,N}^{m}(T)$ for all $s\in\NN$. Then, the sequence of derivatives $h_s^{(k)}$ is equibounded in $[0,T]$ for every $k=1,\ldots,m$.
\end{lemma}
\begin{proof}
    For a contradiction, suppose that for some $k\in \{1,\ldots,m \}$ the derivatives $h_s^{(k)}$ are not equibounded and let $\ell$ be the largest such value of $k$. Without loss of generality, suppose then that 
    $\lim_{s\to\infty} \bar h_s^\ell = \infty$
    with $\bar h_s^\ell := \sup_{t\in [0,T]} |h_s^{(\ell)}(t)|$
    and that $h_s^{(k)}$ are equibounded for $k=\ell+1,\ldots,m+1$ (a.e. if $k=m+1$). Define the normalized sequence
    $
        g_s(t) = h_s^{(\ell)}(t)/\bar h_s^\ell$. 
By construction,
    $
        \sup_{t\in [0,T]} |g_s(t)| = 1$ holds for all $s \in \NN$,
    and since the $h_s^{(\ell+1)}$ are equibounded,
    $
\lim_{s\to\infty} \sup_{t\in [0,T]} |g_s^{(1)}(t)| = 0.
    $
    The sequence $g_s$ is thus equibounded and equicontinuous. By the Arzelà-Ascoli theorem, a subsequence that converges uniformly to a continuous function must exist. Using now $g_s$ to denote this subsequence, due to
$
\lim_{s\to\infty} \sup_{t\in [0,T]} |g_s^{(1)}(t)| = 0
    $,
    the derivatives $g_s^{(1)}$ converge uniformly to the zero function. Then, the limit function must be a nonzero constant $p$, either $p=1$ or $p= -1$. 
Repeated integration yields
    \begin{align*}
        \frac{h_s(t)}{\bar h_s^\ell} = \frac{h_s(0)}{\bar h_s^\ell} + \cdots + \frac{h_s^{(\ell-1)}(0)t^{\ell-1}}{\bar h_s^\ell (\ell-1)!} + \int_0^t \cdots\int_0^{\tau_{\ell-1}} \frac{h_s^{(\ell)}(\tau_\ell)}{\bar h_s^\ell} d\tau_\ell \cdots d\tau_1
    \end{align*}
    Taking limits, 
    the following expression should be true for all $t\in [0,T]$
    \begin{align*}
        \lim_{s\to\infty}\left(\frac{h_s^{(1)}(0)t}{\bar h_s^\ell} + \cdots + \frac{h_s^{(\ell-1)}(0)t^{\ell-1}}{\bar h_s^\ell (\ell-1)!}\right) = -\int_0^t \cdots\int_0^{\tau_{\ell-1}} p\ d\tau_\ell \cdots d\tau_1
        = -p \frac{t^\ell}{\ell !}
    \end{align*}
    and, moreover, the convergence should be uniform within $[0,T]$. This is impossible because polynomials of different degrees are linearly independent.
\end{proof}

\begin{lemma}
\label{prop:gopt:extremal}
    Let $L \in \RR_{\ge 0}$, $N, T \in \RR_{> 0}$, $k,m \in \NN$, and $k \le m$.
    Then,
    \begin{enumerate}[a)]
    \item
    \label{it:extremal:value}
    all $g \in \mathcal{G}_{L,N}^{m}(T)$ satisfying $g^{(k)}(T) = \gopt{k}$ also fulfill $g(T) = N$; and
    \item 
    \label{it:extremal:existence}
    there exists $g \in \mathcal{G}_{L,N}^{m}(T)$ satisfying $g^{(k)}(T) = \gopt{k}$ such that $|g(\tau)| = N$ holds only at finitely many time instants $\tau \in [0,T]$.
    \end{enumerate}
\end{lemma}
\begin{proof}
For item~\ref{it:extremal:value}), suppose to the contrary that $g(T) < N$.
    Then, there exists $\varepsilon > 0$ and an extension $\tilde g \in \mathcal{G}_{L,N}^m(T+\varepsilon)$ of $g$ coinciding with $g$ on the interval $[0,T]$ and satisfying $\tilde g^{(k)}(T + \varepsilon) > g^{(k)}(T)$.
    This yields the contradiction $\gop_{N,k}^{m,L}(T+\varepsilon) > \gop_{N,k}^{m,L}(T)$ to the fact that $\gop_{N,k}^{m,L}$ is non-increasing with respect to $T$.

    In order to show item~\ref{it:extremal:existence}), consider a sequence $h_s \in \mathcal{G}_{L,N}^{m}(T)$ with the property that $\lim_{s\to\infty} h_s^{(k)}(T) = \gopt{k}$. The sequence of derivatives $h_s^{(m)}$ is equibounded according to \cref{lem:G:der:bounded} and has equibounded derivatives by definition of $\mathcal{G}_{L,N}^{m}(T)$. By the Arzel\`a-Ascoli theorem, then the sequence $h_s^{(m)}$ has a uniformly convergent subsequence that we name also $h_s^{(m)}$ without risk of confusion. Next, suppose that the sequence $h_s^{(i)}$ for $i\le m$, which is equibounded by \cref{lem:G:der:bounded}, has equibounded derivatives. Application of Arzel\`a-Ascoli theorem shows that such sequence has a uniformly convergent subsequence that we name also $h_s^{(i)}$. By induction, we arrive to a uniformly convergent subsequence of the original sequence satisfying $\lim_{s\to\infty} h_s^{(k)}(T) = \gopt{k}$ and having uniformly convergent derivatives up to order $m$. Therefore,  existence of at least one $g \in \mathcal{G}_{L,N}^{m}(T)$ satisfying $g^{(k)}(T) = \gopt{k}$ follows. 

    If $L = 0$, then $g$ is a nonconstant polynomial which cannot satisfy $|g(\tau)| = N$ at infinitely many time instants.
    Hence, suppose that $L > 0$.
    If $g$ satsifies $|g(\tau)| = N$ at infinitely many time instants $\tau \in [0,T]$, then these time instants have a largest cluster point $\bar \tau \in [0,T]$.
    Suppose $\lim_{\tau \to \bar\tau} g(\tau) = N$ without restriction of generality.
    Since all derivatives of $g$ up to order $m$ are continuous and uniformly bounded on $[0,T]$, they each have zeros at the stationary points of the their preceding derivative.
    By induction over the derivative order $\ell$, those zeros of $g^{(\ell)}$ are seen to cluster at $\bar \tau$ for each $\ell$.
    Thus, continuity yields $g(\bar \tau) = N$ and $g^{(\ell)}(\bar \tau) = 0$ for $\ell =1, \ldots, m$.
    As a consequence, $\bar \tau < T$ holds,  because $|g^{(k)}(T)| = \gopt{k} > 0$.
    It is thus possible to select $\tilde \tau \in [\bar \tau, T)$ such that $|g(\tau)| = N$ holds only at finitely many time instants in the interval $[\tilde \tau, T]$, such that $g(\tilde \tau) = N$, and such that $|g^{(\ell)}(\tilde \tau)|$ for $\ell = 2, \ldots, m$ (note that $g^{(1)}(\tilde \tau) = 0$ due to $g(\tilde \tau) = N$) are so small that the following construction is possible:
    Define $\tilde g(\tau) = g(\tau)$ on $[\tau_j, T]$ and continuously extend that function, maintaining continuous derivatives up to order $m$, to the time interval $[0,T]$ by means of an appropriate polynomial spline satisfying $\tilde g(\tau) \in [0,N)$ and $\tilde g^{(m+1)}(\tau) \in [-L,L]$ on $[0,\tau_j)$.
    Then, by construction, $\tilde g \in \FL^m$ is still extremal, i.e., $\tilde g^{(k)}(T) = \gopt{k}$, and it satisfies $|\tilde g(\tau)| = N$ only at finitely many time instants, which all lie in the interval $[\tilde\tau, T]$.
\end{proof}

\begin{example}
\label{exmp:gop1}
    Consider the case of first-order differentiation, i.e., $m = 1$, $k = 1$.
    Then, the following can be obtained by direct calculation:
    \begin{align}
    \label{eq:exa:barg:N11L}
    \bar g_{N,1}^{1,L}(t) &= \begin{cases}
            \frac{Lt}{2} + \frac{2N}{t} & t \in [0,2\sqrt{N/L}) \\
            2 \sqrt{NL} & t \ge 2 \sqrt{N/L},
        \end{cases} \\\displaybreak[0]
        \underline g_{N,1}^{1,L}(t) &= \begin{cases}
            Lt & t \in [0,\sqrt{2N/L}) \\
            \sqrt{2L^2 t^2 + 4 N L} - L t & t \in [\sqrt{2N/L}, 4\sqrt{N/L}) \\
            2 \sqrt{NL} & t \ge 4 \sqrt{N/L}.
        \end{cases}
    \end{align}
    for $L > 0$, and $\bar g_{N,1}^{1,0}(t) = \frac{2N}{t}$, $\underline g_{N,1}^{1,0}(t) = 0$.
\end{example}

\subsection{Lower Bounds for Causal Differentiators}

This section shows lower bounds on the worst-case error for any causal differentiator.

In particular, the following proposition is proven at the end of the subsection.
It gives general lower bounds on the worst-case differentiation error, valid for arbitrary initial condition bounds $\R \in \RR_{\ge 0}^{m+1}$, and also tighter lower bounds when restrictions on $\R$ are present.
To state it, define
    \begin{equation}
        \label{eq:ics:lowerbound}
        \bar{\g}^{m,L}_{N}(t) = \begin{bmatrix}
            N & \bar g_{N,1}^{m,L}(t) & \ldots & \bar g_{N,m}^{m,L}(t)
        \end{bmatrix}\TT
    \end{equation}
for all $L,N \ge 0$ and $t > 0$.

\begin{proposition}
\label{prop:causal:worstcase}
    Let $L,N \in \RR_{\ge 0}$, $T \in \RR_{> 0}$, $k,m\in\NN$, $k\le m$ and consider a causal differentiator $\Diff$ of order $m$.
    Then, its worst-case error satisfies
    \begin{enumerate}[a)]
        \item \label{item:prop:causal:worstcase:limit}
        $M_{N,k}^{\FLR^m}(t) \ge \lim_{\tau \to \infty} \bar g_{N,k}^{m,L}(\tau)$ for all $\R \in \RR_{\ge 0}^{m+1}$ and $t \in \RR_{\ge 0}$;
        \item \label{item:prop:causal:worstcase:general}
        $M_{N,k}^{\FLR^m}(t) \ge \bar g_{N,k}^{m,L}(t)$ for all $\R \ge \bar \g^{m,L}_{N}(\theta)$, all $t \ge \theta$ and each $\theta \in \RR_{> 0}$;
        \item \label{item:prop:causal:worstcase:everyic}
        $M_{N,k}^{\FL^m}(t) \ge \bar g_{N,k}^{m,L}(t)$ for all $t \in \RR_{> 0}$.
    \end{enumerate}
\end{proposition}
\cref{prop:causal:worstcase}\ref{item:prop:causal:worstcase:general}) shows that, for sufficiently large initial conditions, $\bar g_{N,k}^{m,L}(t)$ bounds the worst-case error from below as a function of time.
Although the lower bound on the initial condition, namely $\bar\g^{m,L}_{N}$, does not tend to zero as time tends to infinity, \cref{prop:causal:worstcase}\ref{item:prop:causal:worstcase:limit}) shows that the lower bound on the worst-case error is nonetheless true for any initial condition as $t$ tends to infinity.

In order to prove \cref{prop:causal:worstcase}, the following lemma is instrumental.
It gives a simple lower bound on the worst-case differentiation error that is related to the existence of specific bounded functions in $\FLR^m$. \begin{lemma}
\label{lem:causal:worst-case-f}
    Let $L, N, \theta \in \RR_{\ge 0}$, $m\in\NN$, $\R \in \RR_{\ge 0}^{m+1}$ and consider a causal differentiator $\Diff$ of order $m$.
    Suppose that there exists $f \in \FLR^m$ satisfying $|f(\tau)| \le N$ for all $\tau \in [0, \theta]$.
    Then, $M_{N,k}^{\FLR^m}(t) \ge |f^{(k)}(\theta)|$ holds for all $k \in \{1, \ldots, m\}$ and all $t \in [0, \theta]$.
\end{lemma}
\begin{proof}
Define $\eta \in \EN$ as $\eta(\tau) = -f(\tau)$ for $\tau \in [0, \theta]$ and as $\eta(\tau) = N$ otherwise.
Note that then $f(\tau)+\eta(\tau) = 0$ for all such $\tau \in [0, \theta]$.
Define $\gamma = [\Diff_k(f+\eta)](\theta)$.
Clearly, $M_{N,k}^{\FLR^m}(t) \ge |\gamma - f^{(k)}(\theta)|$.
From the conditions on $f$ and $\eta$, causality implies that also $[\Diff_k(-f-\eta)](\theta) = \gamma$.
Since $-f \in \FLR^m$, $-\eta \in \EN$ as well, this yields $M_{N,k}^{\FLR^m}(\theta) \ge \max\{ |\gamma - f^{(k)}(\theta)|, |\gamma+f^{(k)}(\theta)| \} \ge |f^{(k)}(\theta)|$, proving the claim.
\end{proof}    

We are now in the position to prove \cref{prop:causal:worstcase}.

\begin{proof}[Proof of \cref{prop:causal:worstcase}]
    \ref{item:prop:causal:worstcase:limit})
    As a consequence of \cref{lem:gopt-limit}\ref{it:gopt-limit:limits}), it suffices to show $M_{N,k}^{\FLR^m}(t) \ge \lim_{\tau \to \infty} \underline g_{N,k}^{m,L}(\tau)$.
    Hence, suppose for a contradiction that the relation $M_{N,k}^{\FLR^m}(t) = \beta < \lim_{\tau \to \infty} \underline g_{N,k}^{m,L}(\tau)$ holds for some $\R\in\RR_{\ge 0}^{m+1}$ and $t\in\RRge$.
    Then, there exists a sufficiently large $\theta \ge t$ and an $f \in \mathcal{G}_{L,N}^{m}(\theta) \cap \F_{L,\bm{0}}^m \subseteq \FLR^m$ with $f^{(k)}(\theta) > \beta$.
    Lemma~\ref{lem:causal:worst-case-f} then yields the contradiction $M_{N,k}^{\FLR^m}(t) \ge |f^{(k)}(\theta)| > \beta$.

\ref{item:prop:causal:worstcase:general}) Let $\bar \R(t) = \bar\g^{m,L}_{N}(t)$.
It is sufficient to show $M_{N,k}^{\F_{L,\bar \R(t)}^m}(t) \ge \bar g_{N,k}^{m,L}(t)$ for all $t > 0$.
Suppose to the contrary that $M_{N,k}^{\F_{L,\bar \R(t)}^m}(t) = \beta < \bar g_{N,k}^{m,L}(t)$ for some $t>0$.
Consider any $f \in \mathcal{G}_{L,N}^{m}(t)$ with $f^{(k)}(t) > \beta$.
Note that $f \in \F_{L,\bar \R(t)}^m$, because \eqref{eq:def:gopt} is symmetric with respect to time reversal, and hence $|f^{(i)}(0)| \le \bar g_{N,i}^{m,L}(t)$ for $i=1, \ldots, m$ in addition to $|f(0)| \le N$.
Since $|f(\tau)| \le N$ for all $\tau \in [0, t]$, Lemma~\ref{lem:causal:worst-case-f} then yields the contradiction $M_{N,k}^{\F_{L,\bar \R(t)}^m} \ge |f^{(k)}(t)| > \beta$.

\ref{item:prop:causal:worstcase:everyic}) This follows from $M_{N,k}^{\FL^m}(t) \ge M_{N,k}^{\FLR^m}(t)$ for every $\R\in\RRge^{m+1}$ and \cref{item:prop:causal:worstcase:general}).
\end{proof}

\subsection{Lower Bounds for Exact Differentiators}

Exact differentiators produce an output that equals the real derivative of the input signal in the absence of measurement noise, provided that some higher-order derivative bound is known to the differentiator.
This subsection provides lower bounds for the worst-case error of such differentiators.

In doing so, one may distinguish the class of signals over which the differentiator is exact, denoted by $\F_{M}^m$ in the following, and the class $\FL^m$ over which the worst-case error is considered.
In particular, the following proposition will be proven at the end of the subsection.
\begin{proposition}
\label{prop:exact:worstcase}
    Let $L,M,N \in \RR_{\ge 0}$, $T \in \RR_{> 0}$, $k,m \in \NN$, $k\le m$ and consider a causal differentiator $\Diff$ of order $m$.
    Suppose that $\Diff$ is exact in finite time over $\F_{M}^m$ and recall \eqref{eq:ics:lowerbound}.
    Then, its worst-case error satisfies
    \begin{enumerate}[a)]
        \item \label{item:prop:exact:worstcase:limit}
        $M_{N,k}^{\FLR^m}(t) \ge \lim_{\tau \to \infty} \bar g_{N,k}^{m,L+M}(\tau)$ for all $\R \in \RR_{\ge 0}^{m+1}$ and $t \in \RR_{\ge 0}$.
\end{enumerate}

    If $\Diff$ is additionally exact from the beginning over $\F_M^m$, then
    \begin{enumerate}[a)]\setcounter{enumi}{1}
        \item \label{item:prop:exact:worstcase:general}
        $M_{N,k}^{\FLR^m}(t) \ge \bar g_{N,k}^{m,L+M}(t)$ for all $\R \ge \bar \g^{m,L+M}_{N}(\theta)$ and all $t \ge \theta$, for each $\theta \in \RR_{> 0}$;
        \item \label{item:prop:exact:worstcase:beginning}
        $M_{N,k}^{\FL^m}(t) \ge \bar g_{N,k}^{m,L+M}(t)$ for all $t \in \RR_{> 0}$. \end{enumerate}
\end{proposition}

For proving \cref{prop:exact:worstcase}, the following lemma is instrumental.
Analogously to \cref{lem:causal:worst-case-f}, it provides a simple lower bound on the worst-case differentiation error related to specific bounded functions but in the case when exact differentiation is possible.
\begin{lemma}
\label{lem:exact:worst-case-f}
    Let $L, M, N \in \RR_{\ge 0}$, $\theta \in \RRgt$, $m\in\NN$, $\R \in \RR_{\ge 0}^{m+1}$ and consider a causal differentiator $\Diff$.
    Suppose $\M_{0}^{\F_{M,\R}^m}(\theta) = \bm{0}$ and that there exists $g \in \F_{L+M,\R}^m$ satisfying $|g(\tau)| \le N$  for all $\tau \in [0, \theta]$.
    Then, $M_{N,k}^{\FLR^m}(t) \ge |g^{(k)}(\theta)|$ for all $k \in \{1, \ldots, m\}$ and all $t \in [0, \theta]$.
\end{lemma}
\begin{proof}
    If $L=M=0$, the result follows from \cref{lem:causal:worst-case-f}. Otherwise, let $f = - \frac{L}{L+M} g \in \FLR^m$ and define $\eta(\tau) = g(\tau)$ for $\tau \in [0, \theta]$ and $\eta(\tau) = N$ otherwise.
    Clearly, $\eta \in \EN$ and $f(\tau)+\eta(\tau) = \frac{M}{L+M} g(\tau)$ holds for all $\tau \in [0, \theta]$.
    Since $\Diff$ is causal, $[\Diff_k(f+\eta)](\theta) = [\Diff_k(\frac{M}{L+M} g )](\theta) = \frac{M}{L+M} g^{(k)}(\theta)$ is then obtained, because $\frac{M}{L+M} g \in \F_{M,\R}^m$ and $\M_{0}^{\F_{M,\R}^m}(\theta) = \bm{0}$, i.e., $\Diff$ differentiates this input exactly at and after time $\theta$.
    Consequently, $M_{N,k}^{\FLR^m}(t) \ge |f^{(k)}(\theta) - [\Diff_k(f+\eta)](\theta)| = |g^{(k)}(\theta)|$.
\end{proof}

We can now prove \cref{prop:exact:worstcase}.

\begin{proof}[Proof of \cref{prop:exact:worstcase}]
    \ref{item:prop:exact:worstcase:limit})
    As a consequence of \cref{lem:gopt-limit}\ref{it:gopt-limit:limits}), it suffices to show $M_{N,k}^{\FLR^m}(t) \ge \lim_{\tau \to \infty} \underline g_{N,k}^{m,L+M}(\tau)$.
    Hence, assume to the contrary that the relation $M_{N,k}^{\FLR^m}(t) = \beta < \lim_{\tau \to \infty} \underline g_{N,k}^{m,L+M}(\tau)$ holds for some $\R \in \RR_{\ge 0}^{m+1}$ and $t\in\RRge$.
    Then, there exists a sufficiently large $\theta \ge t$ and a $g \in \mathcal{G}_{L+M,N}^{n}(\theta) \cap \F_{L+M,\bm{0}}^m \subseteq \F_{L+M,\R}^m$ with $g^{(k)}(\theta) > \beta$.
    The contradiction $M_{N,k}^{\FLR^m}(t) \ge |g^{(k)}(\theta)| > \beta$ is then obtained by applying \cref{lem:exact:worst-case-f}. 

    \ref{item:prop:exact:worstcase:general})
    Let $\bar \R(t) = \bar\g^{m,L+M}_{N}(t)$. 
It is sufficient to show $M_{N,k}^{\F_{L,\bar \R(t)}^m}(t) \ge \bar g_{N,k}^{m,L+M}(t)$ for all $t > 0$.
Suppose to the contrary that $M_{N,k}^{\F_{L,\bar\R(t)}^m}(t) = \beta < \bar g_{N,k}^{m,L+M}(t)$.
Consider any $g \in \mathcal{G}_{L+M,N}^{m}(t)$ with $g^{(k)}(t) > \beta$.
Note that $g \in \F_{L+M,\bar \R(t)}^m$, because \eqref{eq:def:gopt} is symmetric with respect to time reversal.
Hence, $|g^{(i)}(0)| \le \bar g_{N,i}^{m,L+M}(t)$ holds for $i=1, \ldots, m$ in addition to $|g(0)| \le N$.
Since $|g(\tau)| \le N$ for all $\tau \in [0, t]$, \cref{lem:exact:worst-case-f} then yields the contradiction $M_{N,k}^{\F_{L,\bar\R(t)}^m} \ge |g^{(k)}(t)| > \beta$.

    \ref{item:prop:exact:worstcase:beginning}) 
    This follows from $M_{N,k}^{\FL^m}(t) \ge M_{N,k}^{\FLR^m}(t)$ for all $\R \in \RRge^{m+1}$ and \cref{item:prop:exact:worstcase:general}).
\end{proof}

\section{Deterministically Optimal Differentiation}
\label{sec:opt-diff}

The previous section provided lower bounds on the worst-case differentiation error for causal and exact differentiators. This section addresses suitable definitions for best (i.e. lowest) worst-case error and consequently of optimal differentiators in a deterministic sense, as well as the existence of such types of differentiators. 

\subsection{Definition and Properties of Optimal Differentiators}

\subsubsection{Best Worst-Case Error of Causal Differentiators}

As a first step towards defining deterministically optimal differentiators, \cref{prop:causal:worstcase} motivates the following definition of a causal differentiator having best worst-case error with respect to a given noise bound $N$.
\begin{defin}
\label{def:optimal:causal}
    Let $L,N\in\RRge$, $m\in\NN$. A causal differentiator $\Diff$ of order $m$ is said to have best worst-case error over $\FL^m$ with respect to noise in $\EN$, if its worst-case error satisfies
    $M_{N,k}^{\FLR^m}(t) \le \bar g_{N,k}^{m,L}(t)$ for all $\R \in \RR_{\ge 0}^{m+1}$, $t \in \RR_{> 0}$, and $k\in\NN$ satisfying $k\le m$.
\end{defin}

Note that, while not considered here explicitly, it is also possible to define the best-worst case error property for certain selected, individual derivative orders.
The following results will also be useful in that case.

The following proposition shows some important properties that all differentiators with best worst-case error necessarily have.
It will be proven at the end of the subsection, after two auxiliary lemmata.

\begin{proposition}
\label{prop:bestworstcase:properties}
    Let $L, N, L', N' \in \RR_{\ge 0}$, $m\in\NN$ and consider a causal differentiator $\Diff$ of order $m$.
    Suppose that $\Diff$ has best worst-case error over $\FL^m$ with respect to noise in $\EN$.
    Then,
    \begin{enumerate}[a)]
    \item
    \label{it:bestworstcase:homogeneity}
        $\Diff$ has best worst-case error over $\F_{\alpha L}^m$ with respect to $\E_{\alpha N}$ for all $\alpha \in [0,1]$;
        \item 
        \label{it:bestworstcase:exact0}
        $\Diff$ is exact from the beginning over $\F_0^m$;
\item
        \label{it:bestworstcase:robustness}
        $\Diff$ is robust almost from the beginning over $\F_0^m$, if $N > 0$;
        \item 
        \label{it:bestworstcase:exactnecessary}
        $\Diff$ is not exact over $\FL^m$, if $N >0$ and $L > 0$;
        \item 
        \label{it:bestworstcase:parameter}
        $\Diff$ can have best worst-case error over $\F_{L'}^m$ with respect to $\E_{N'}$ only if $L N' = L' N$.
    \end{enumerate}
\end{proposition}

\begin{remark}
\label{rem:bestworstcase:exactness}
    Regarding item~\ref{it:bestworstcase:exactnecessary}), note that, as an immediate consequence of Definition~\ref{def:optimal:causal}, $\Diff$ indeed is exact from the beginning over $\FL^m$ if and only if it has best worst-case error over $\FL^m$ with respect to noise in $\E_0$.
\end{remark}
\begin{remark}\label{rem:NLratio}
Item~\ref{it:bestworstcase:parameter}) essentially means that the ratio $N/L$ has to be known to any differentiator that has best worst-case error over $\FL^m$ with respect to noise in $\EN$.
\end{remark}

In order to prove the properties of differentiators with best worst-case error in \cref{prop:bestworstcase:properties}, two auxiliary lemmata are now proven.
These lemmata allow to relate the differentiation error and the best worst-case error property, respectively, of a given differentiator between different values of $L$ and $N$.

\begin{lemma}
\label{lem:worstcase:additive}
    Let $L_1, L_2, N_1, N_2, \theta, c \in \RR_{\ge 0}$, $k,m \in \NN$, $k\le m$, and consider a causal differentiator $\Diff$ of order not less than $k$.
    Suppose that $f_1 \in \F_{L_1}^m$, $\eta_1 \in \E_{N_1}$ exist such that $|[\Diff_k(f_1+\eta_1)](\theta) - f_1^{(k)}(\theta)| > c$.
    Then, $f\in \F_{L_1+L_2}^m$, $\eta \in \E_{N_1+N_2}$ exist such that $|[\Diff_k(f+\eta)](\theta) - f^{(k)}(\theta)| > c + \gop_{N_2,k}^{m,L_2}(\theta)$.
\end{lemma}
\begin{proof}
Let $s = \sign( [\Diff_k(f_1+\eta_1)](\theta) - f_1^{(k)}(\theta) )$, choose $\varepsilon > 0$ such that $|[\Diff_k(f_1+\eta_1)](\theta) - f_1^{(k)}(\theta))| \ge (c + \varepsilon)$, and pick a $g \in \mathcal{G}_{L_2,N_2}^{m}(\theta)$ which satisfies $|g^{(k)}(\theta)| > \gop_{N_2,k}^{m,L_2}(\theta) - \varepsilon$ and $\sign(g^{(k)}(\theta)) = s$.
    Define $f = f_1 - g \in \F_{L_1+L_2}^n$, and $\eta \in \E_{N_1+N_2}$ as $\eta(t) = 0$ for $t > \theta$ and $\eta(t) = \eta_{1}(t) + g(t)$  otherwise.
    Since $f(t)+\eta(t) = f_{1}(t) + \eta_{1}(t)$ for $t \in [0,\theta]$, causality then yields
    \begin{equation}
        [\Diff_{k}(f+\eta) ](\theta) s \ge f_1^{(k)}(\theta)s +c + \varepsilon = (f^{(k)}(\theta) + g^{(k)}(\theta))s + c + \varepsilon .
    \end{equation}
    Subtracting $s f^{(k)}(\theta)$ and recalling the signs,
    \begin{equation}
        |[\Diff_k(f+\eta)](\theta) - f^{(k)}(\theta)| \ge |g^{(k)}(\theta)| + c + \varepsilon > c + \gop_{N_2,k}^{m,L_2}(\theta)
    \end{equation}
    completing the proof.
\end{proof}

\begin{lemma}
\label{prop:optimal:causal:uniqueness}
    Let $L \in \RR_{> 0}$, $N_1, N_2, N_3 \in \RR_{\ge 0}$, $N_1 < N_2 < N_3$, $m\in\NN$ and consider a causal differentiator $\Diff$ of order $m$.
    Then, at most one of the following statements can be true:
    \begin{enumerate}[a)]
        \item \label{item:bestLN1}
        $\Diff$ has best worst-case error over $\FL^m$ with respect to noise in $\E_{N_1}$
        \item \label{item:bestLN2}
        $\Diff$ has best worst-case error over $\FL^m$ with respect to noise in $\E_{N_2}$
        \item \label{item:best0N3}
        $\Diff$ has best worst-case error over $\F_{0}^m$ with respect to noise in $\E_{N_3}$
    \end{enumerate}
\end{lemma}
\begin{remark}
\label{rem:optimalL0}
    Note that the statement is false for $L = 0$.
Indeed, the worst-case error of the differentiator defined as
$\Diff[u](t) = \frac{u(t) - u(0)}{t}$ for $t > 0$
    has worst-case error $M_{N,1}^{\F_{0,\R}^1}(t) \le \frac{2N}{t} = \bar g_{N,1}^{1,0}(t)$ for all $N \in \RR_{\ge 0}$ and $t > 0$.
    It hence has best worst-case error over $\F_0^1$ with respect to noise in $\EN$ for every $N \ge 0$.
\end{remark}
\begin{proof}
To see that \ref{item:bestLN1}) and \ref{item:bestLN2}) preclude each other, suppose to the contrary that $\Diff$ has best worst-case error over $\FL^m$ with respect to noise in both $\E_{N_1}$ and $\E_{N_2}$.
Let $N = \frac{N_1+N_2}{2}$.
    Due to Lemma~\ref{lem:gopt:concavelimit}, it is then possible to pick a sufficiently small $\varepsilon > 0$ and sufficiently large $T > 0$ such that $\bar g_{N,k}^{m,L}(T) \ge \frac{1}{2} \bar g_{N_1,k}^{m,L}(T) + \frac{1}{2} \bar g_{N_2,k}^{m,L}(T) + \varepsilon$.
    Consider any $g \in \mathcal{G}_{L,N}^{m}(T)$ with $g^{(k)}(T) \ge \bar g_{N,k}^{m,L}(T) - \frac{\varepsilon}{2}$ and let $u = \frac{N_2-N_1}{N_2+N_1} g$.
    Since $|u(\tau) - g(\tau)| = \frac{2N_1}{N_1+ N_2} |g(\tau)| \le N_1$ and $|u(\tau) + g(\tau)| = \frac{2N_2}{N_1+ N_2} |g(\tau)| \le N_2$ hold for all $\tau \in [0,T]$, the differentiator output, when applying $u$ as an input, fulfills the inequalities $\abs{[\Diff_k u](T) - g^{(k)}(T)} \le \bar g_{N_1,k}^{m,L}(T)$ and $\abs{[\Diff_k u](T) + g^{(k)}(T)} \le \bar g_{N_2,k}^{m,L}(T)$ due to the best worst-case error property. As a consequence, 
    $g^{(k)}(T) - \bar g_{N_1,k}^{m,L}(T) \le [\Diff_k u](T) \le -g^{(k)}(T) + \bar g_{N_2,k}^{m,L}(T)$ holds, and rearranging yields the contradiction
    \begin{equation}
        \bar g_{N_1,k}^{m,L}(T) + \bar g_{N_2,k}^{m,L}(T) \ge 2g^{(k)}(T) \ge 2\bar g_{N,k}^{m,L}(T) - \varepsilon \ge \bar g_{N_1,k}^{m,L}(T) + \bar g_{N_2,k}^{m,L}(T) + \varepsilon.
    \end{equation}
    To show that \ref{item:bestLN1}) and \ref{item:best0N3}) preclude each other, suppose to the contrary that both are true and select  $\varepsilon, T > 0$, such that $\gop_{N_3,k}^{m,0}(T) + \gop_{N_1,k}^{m,0}(T) < \gop_{N_3,k}^{m,L}(T) - \gop_{N_1,k}^{m,L}(T) - \varepsilon$, which is possible because the left-hand side tends to zero as $T$ increases, while the right-hand side has a positive limit for $N_3 > N_1$.
    Pick $g_2 \in \mathcal{G}_{0,N_1}^{m}(T)$, $g_3 \in \mathcal{G}_{L,N_3}^{m}(T)$ such that $g_3^{(k)}(T) > \gop_{N_3,k}^{m,L}(T) - \varepsilon$, and consider $u = g_2 + g_3$.
    Due to the assumed worst-case error properties, $[\Diff u](T) \le g_2^{(k)}(T) + \gop_{N_3,k}^{m,0}(T) \le \gop_{N_1,k}^{m,0}(T) + \gop_{N_3,k}^{m,0}(T)$ then holds, as well as $[\Diff u](T) \ge g_3^{(k)}(T) - \gop_{N_1,k}^{m,L}(T)$. This leads to the contradiction
    \begin{equation}
    \gop_{N_1,k}^{m,0}(T) + \gop_{N_3,k}^{m,0}(T) \ge g_3^{(k)}(T) - \gop_{N_1,k}^{m,L}(T) > \gop_{N_3,k}^{m,L}(T) - \gop_{N_1,k}^{m,L}(T) - \varepsilon.
    \end{equation}
    The fact that \ref{item:bestLN2}) and \ref{item:best0N3}) preclude each other, finally, follows from the fact that \ref{item:bestLN1}) and \ref{item:best0N3}) preclude each other and that \ref{item:bestLN2}) has the same structure as \ref{item:bestLN1}).
\end{proof}

\begin{proof}[Proof of \cref{prop:bestworstcase:properties}]
    For item~\ref{it:bestworstcase:homogeneity}), suppose to the contrary that there exist $\alpha \in [0,1]$, $f_1 \in \F_{\alpha L}^m$, $\eta_1 \in \E_{\alpha N}$ such that $|[\Diff_k(f_1+\eta_1)](T) - f_1^{(k)}(T)| > \gop_{\alpha N,k}^{m,\alpha L}(T) = \alpha \gop_{N,k}^{m,L}(T)$ holds for some $T \in \RR_{> 0}$, where \eqref{eq:gopt-homogeneity} of \cref{lem:gopt-scaling} was used.
Using \cref{lem:worstcase:additive} with $L_1 = \alpha L$, $L_2 = (1-\alpha) L$, $N_1 = \alpha N$, $N_2 = (1-\alpha) N$, and \eqref{eq:gopt-homogeneity} again, then guarantees existence of $f \in \FL^m$, $\eta \in \EN$ such that
\begin{align}
|[\Diff_k(f+\eta)](T) - f^{(k)}(T)| &> \alpha \gop_{N,k}^{m,L}(T) + (1-\alpha) \gop_{N,k}^{m,L}(T) = \gop_{N,k}^{m,L}(T)
\end{align}
contradicting the fact that $\Diff$ has best worst-case error over $\FL^m$ with respect to noise in $\EN$.
Item~\ref{it:bestworstcase:exact0}) follows by setting $\alpha = 0$ in the previous item, and noting that then $M^{\F_0^m}_{0,k}(t) \le \gop_{0,k}^{m,0}(t) = 0$ holds for all $t > 0$.
To prove item~\ref{it:bestworstcase:robustness}) note that for $\alpha \in [0,1]$, 
\begin{equation}
    Q^{\F_0^m}_{\alpha N,k}(t) \le M^{\F_0^m}_{\alpha N,k}(t) + M^{\F_0^m}_{0,k}(t) \le M^{\F_{\alpha L}^m}_{\alpha N,k}(t) \le \alpha \gop_{N,k}^{m,L}(t)
\end{equation}
holds for $t > 0$, where items~\ref{it:bestworstcase:homogeneity}) and~\ref{it:bestworstcase:exact0}) are used.
Hence, $Q^{\F_0^m}_k(t) \le \lim_{\alpha \to 0^+} \alpha \gop_{N,k}^{m,L}(t) = 0$.
For item~\ref{it:bestworstcase:exactnecessary}), let $N,L > 0$ and suppose that $\Diff$ is exact in finite time over $\FL^m$.
Then \cref{prop:exact:worstcase}, item~\ref{item:prop:exact:worstcase:limit}) yields\footnote{the strict inequality is not true if either $L$ or $N$ equal $0$.} $M_{N,k}^{\FL^m}(t) \ge \lim_{\tau \to \infty} \bar g_{N,k}^{m,2L}(\tau) > \lim_{\tau \to \infty} \bar g_{N,k}^{m,L}(\tau)$, contradicting  the bound $M_{N,k}^{\FL^m}(t) \le \bar g_{N,k}^{m,L}(t)$ for all $t > 0$.
To see item~\ref{it:bestworstcase:parameter}), suppose that $\Diff$ has best worst-case error over $\F_{L'}^m$ with respect to noise in $\E_{N'}$ in addition to having best worst-case error over $\FL^m$ with respect to noise in $\EN$.
Without restriction of generality, assume $L' \le L$.
Distinguish the cases $L' = 0$ and $L' > 0$.
In the first case, with the purpose of provoking a contradiction, suppose that $L N' > 0$.
Then, there exists $\alpha \in (0,1]$ such that $\alpha N < N'$.
Using item~\ref{it:bestworstcase:homogeneity}) with that $\alpha$ and applying \cref{prop:optimal:causal:uniqueness} with $N_2 = \alpha N$, $N_3 = N'$ then yields a contradiction.
In the second case, i.e., $L' > 0$, use item~\ref{it:bestworstcase:homogeneity}) with $\alpha = \frac{L'}{L}$ to obtain a contradiction by applying \cref{prop:optimal:causal:uniqueness} with $N_1 = \min\{\alpha N, N'\}$, $N_2 = \max\{\alpha N, N'\}$ unless $N_1 = N_2$, i.e., unless $L' N = L N'$.
\end{proof}

\subsubsection{Best Worst-Case Error of Exact Differentiators and Deterministic Optimality}

The previous results show that a differentiator with best worst-case error necessarily requires knowledge of the noise amplitude $N$ in the form of the ratio $N/L$ (cf. \cref{rem:NLratio}).
We now turn toward defining a suitable concept of optimality of a differentiator irrespective of the noise amplitude for a given $L > 0$.
Ideally, such an optimal differentiator would be required to have best worst-case error over $\FL^m$ with respect to $\EN$ for all $N \ge 0$.
However, this is impossible according to \cref{prop:optimal:causal:uniqueness}.
Indeed, having best worst-case error over $\FL^m$ with respect to $\E_0$ is equivalent to exactness from the beginning (cf. \cref{rem:bestworstcase:exactness}), which precludes having best worst-case error for $L, N > 0$ according to \cref{prop:bestworstcase:properties}\ref{it:bestworstcase:exactnecessary}).
Given that exactness is a not only desired but also useful property of a differentiator, the previous facts motivate some definition of optimality of the worst-case error that can be satisfied within the class of exact differentiators.
\cref{prop:exact:worstcase} hence motivates the following definition of deterministically optimal causal differentiators.

\begin{defin}
\label{def:optimal:exact}
    Let $L\in\RRge$, $m\in\NN$. A causal differentiator $\Diff$ of order $m$ is called deterministically optimal over $\FL^m$, if its worst-case error satisfies
    $M_{N,k}^{\F_{L'}^m}(t) \le \bar g_{N,k}^{m,L+L'}(t)$ for all $t \in \RR_{> 0}$, $L' \in [0, L]$, $N \in \RR_{\ge 0}$ and $k\in\{1,\ldots,m\}$.
\end{defin}

The following proposition shows that it is indeed enough to impose the definition's condition for $L' = L$.
\begin{proposition}
    \label{prop:optimal:exact:equivdef}
    Let $L \in \RR_{\ge 0}$, $m \in \NN$ and consider a causal differentiator $\Diff$ of order $m$.
    Then, $\Diff$ is deterministically optimal over $\FL^m$ if and only if its worst-case error satisfies
    \begin{equation}
    \label{eq:def:optimal:exact}
    M_{N,k}^{\FL^m}(t) \le \gop_{N,k}^{m,2L}(t)
    \end{equation}
    for all $t \in \RR_{> 0}$, $N \in \RR_{\ge 0}$, and $k\in\{1,\ldots,m\}$.
\end{proposition}
\begin{proof}
It is clear that \eqref{eq:def:optimal:exact} is necessary for \cref{def:optimal:exact}.
To show that it is also sufficient, suppose to the contrary that there exist $N \in \RR_{\ge 0}$, $\alpha \in [0,1]$, $f_1 \in \F_{\alpha L}^m$, $\eta_1 \in \EN$ such that $|[\Diff_k(f_1+\eta_1)](T) - f_1^{(k)}(T)| > \gop_{N,k}^{m,(1+\alpha)L}(T)$ holds for some $T \in \RR_{> 0}$ and $k\in\NN$ with $k\le m$.
Using \cref{lem:worstcase:additive} with $L_1 = \alpha L$, $L_2 = (1-\alpha) L$, $N_1 = N$, $N_2 = \frac{1-\alpha}{1+\alpha} N$ then guarantees existence of $f \in \FL^m$, $\eta \in \E_{N'}$ with $N' = N_1 + N_2 = \frac{2}{1+\alpha} N$ such that
\begin{align}
|[\Diff_k(f+\eta)](T) - f^{(k)}(T)| &> \gop_{N,k}^{m,(1+\alpha)L}(T) + \gop_{\frac{1-\alpha}{1+\alpha} N,k}^{m,(1-\alpha)L}(T) \nonumber \\
&= (1+\alpha) \gop_{\frac{1}{1+\alpha} N,k}^{m,L}(T) + (1-\alpha) \gop_{\frac{1}{1+\alpha} N,k}^{m,L}(T) \nonumber \\
&= 2 \gop_{\frac{1}{1+\alpha} N,k}^{m,L}(T) = \gop_{N',k}^{m,2L}(T),
\end{align}
where \eqref{eq:gopt-homogeneity} in \cref{lem:gopt-scaling} was used. This contradicts the fact that \eqref{eq:def:optimal:exact} holds for all $N\in\RRge$. \end{proof}

A differentiator that has best worst-case error over $\FL^m$ with respect to noise in $\EN$ for all $N \ge 0$ (which is possible only for $L = 0$, cf. Remark~\ref{rem:optimalL0}) is also deterministically optimal over $\FL^m$.
However, deterministic optimality of a differentiator over $\FL^m$ does not imply its having best worst-case error over $\FL^m$ with respect to noise in $\EN$, except if $N = 0$ or $L = 0$.
Indeed, causal differentiators that are deterministically optimal are always exact from the beginning and, additionally, robust almost from the beginning over $\FL^m$ (cf. \cref{prop:bestworstcase:properties}\ref{it:bestworstcase:exactnecessary}).
\begin{proposition}
\label{prop:optimal:properties}
    Let $L, L' \in \RR_{\ge 0}$, $m \in \NN$ and suppose that a causal differentiator $\Diff$ of order $m$ is deterministically optimal over $\FL^m$.
    Then,
    \begin{enumerate}[a)]
    \item
    \label{it:det-optimal:exactness}
    $\Diff$ is exact from the beginning over $\FL^m$;
    \item
    \label{it:det-optimal:robustness}
    $\Diff$ is robust almost from the beginning over $\FL^m$;
\item 
    \label{it:det-optimal:parameter}
    $\Diff$ can be deterministically optimal over $\F_{L'}^m$ only if $L' = L$.
    \end{enumerate}
\end{proposition}
\begin{proof}
For $t > 0$ and $k\in\NN$, $k\le m$, Definition~\ref{def:optimal:exact} yields $M_{0,k}^{\FL^m}(t) \le \bar g_{0,k}^{m,2L}(t) = 0$, proving item~\ref{it:det-optimal:exactness}), and
    $
        Q_{N,k}^{\FL^m}(t) \le M_{N,k}^{\FL^m}(t) + M_{0,k}^{\FL^m}(t) \le \bar g_{N,k}^{m,2L}(t).
    $
    Item~\ref{it:det-optimal:robustness}) follows from the fact that $\lim_{N \to 0^+} \bar g_{N,k}^{m,2L}(t) = \bar g_{0,k}^{m,2L}(t) = 0$.
To show item~\ref{it:det-optimal:parameter}), finally, let $L > L'$ without restriction of generality.
    If $\Diff$ is deterministically optimal over $\FL^m$, then it is also exact from the beginning over $\FL^m$ according to item~\ref{it:det-optimal:exactness}).
    For each $N > 0$ and $t > 0$, \cref{prop:exact:worstcase} then yields $\R \in \RR_{\ge 0}^{m+1}$ such that $M_{N,k}^{\F_{L',\R}^m}(t) \ge \bar g_{N,k}^{m,L+L'}(t) > \bar g_{N,k}^{m,2L'}(t)$, precluding deterministic optimality over $\F_{L'}^m$.
\end{proof}

As a consequence of exactness from the beginning over $\F_0^m$ according to \cref{{prop:bestworstcase:properties}}, item~\ref{it:bestworstcase:exact0}) or over $\FL^m$ according to \cref{prop:optimal:properties}, item~\ref{it:det-optimal:exactness}), differentiators that either have best worst-case error over $\FL^m$ in the sense of \cref{def:optimal:causal} or are deterministically optimal in the sense of \cref{def:optimal:exact} can never be robust from the beginning due to \cref{th:connections:causal-exact}.
Indeed, also the differentiator shown in \cref{rem:optimalL0} is robust only \emph{almost} from the beginning.

Item~\ref{it:bestworstcase:parameter}) of \cref{prop:bestworstcase:properties} uses \cref{prop:optimal:causal:uniqueness} to show that designing a differentiator with best worst-case error in the sense of Definition~\ref{def:optimal:causal} with respect to noise in $\EN$ necessarily requires the ratio of $N/L$ of noise amplitude and highest derivative bound to be known to the differentiator.
While such a differentiator always features some limited robustness and exactness properties over $\F_0^m$ according to items~\ref{it:bestworstcase:exact0}) and~\ref{it:bestworstcase:robustness}) of Proposition~\ref{prop:bestworstcase:properties}, it cannot be exact over $\FL^m$ according to  item~\ref{it:bestworstcase:exactnecessary}), unless the noise amplitude is known to be zero (cf. also \cref{rem:bestworstcase:exactness}).
On the other hand, deterministic optimality in the sense of \cref{def:optimal:exact} comes at the cost of increasing the worst-case error, but achieves optimality for all noise amplitudes, in some sense, and yields a robust exact differentiator over $\FL^m$ according to Proposition~\ref{prop:optimal:properties}, items~\ref{it:det-optimal:exactness}) and~\ref{it:det-optimal:robustness}).
Such a differentiator does not require knowledge of the noise amplitude $N$, while
knowledge of $L$ remains required according to Proposition~\ref{prop:optimal:properties}, item~\ref{it:det-optimal:parameter}).

For differentiation order $m = 1$, examples of a linear differentiator having best worst-case error over $\FL^1$ with respect to noise in $\EN$ (\cref{def:optimal:causal}) and a nonlinear differentiator that is deterministically optimal over $\FL^1$ (\cref{def:optimal:exact}) can be obtained from \cite[Lemma~5.4]{seehai_auto23} and \cite[Lemma~1]{aldsee_tac25}, respectively.
\begin{eexample}
        For given $N,L > 0$, consider the first-order linear differentiator $\Diff_T$ defined as
        \begin{equation*}
            [\Diff_T u](t) = \begin{cases}
                0 & \text{if $t = 0$} \\
                \frac{u(t) - u(0)}{t} & \text{if $t \in (0, T)$} \\
                \frac{u(t) - u(t-T)}{T} & \text{otherwise}
            \end{cases}
        \end{equation*}
        with the particular parameter choice $T = 2\sqrt{N/L}$.
        Using \cite[Lemma 5.4]{seehai_auto23} and $\gop_{N,1}^{1,L}$ from \Cref{exmp:gop1}, its differentiation error is bounded by
        \begin{align*}
            |[\Diff_T (f+\eta)](t) - \dot f(t)| &\le \begin{cases}
                \frac{L t}{2} + \frac{2N}{t} & \text{if } t \in (0,T) = (0, 2 \sqrt{N/L} ) \\
                \frac{L T}{2} + \frac{2N}{T} = 2 \sqrt{NL} & \text{if } t \ge 2 \sqrt{N/L}
            \end{cases}
            \nonumber \\
            &= \gop_{N,1}^{1,L}(t)
        \end{align*}
        for $f \in \FL^1, \eta \in \EN$, and hence it has best worst-case error over $\FL^1$ with respect to noise in $\EN$ in the sense of Definition~\ref{def:optimal:causal}.
        It is not deterministically optimal in the sense of Definition~\ref{def:optimal:exact} (cf. also Proposition~\ref{prop:optimal:exact:equivdef}), because $\gop_{N,1}^{1,2L}(t) > \gop_{N,1}^{1,L}(t)$.
\end{eexample}
\begin{eexample}
\label{exmp:ored}
        For given $L > 0$, consider the first-order optimal robust exact differentiator $\Diff_{\mathrm{O}}$, a specific case of the differentiator originally proposed in \cite[Section~5.1]{seehai_auto23}, defined here via the relations
        \begin{subequations}
        \begin{align*}
            \hat N(t) &= \frac{1}{2} \sup_{\substack{T \in (0,t] \\ \sigma \in [0,T]}} \left( \left| u(t-\sigma) - u(t) + \frac{u(t) - u(t-T)}{T} \sigma \right| - \frac{L \sigma (T-\sigma)}{2} \right) \\
            \hat T &= \min\left\{ t, 2 \sqrt{\hat N(t)/L} \right\} \\
            [\Diff_{\mathrm{O}} u](t) &= \begin{cases}
                0 & \text{if $t = 0$} \\
                \lim_{h \to 0^+} \frac{u(t) - u(t-h)}{h} & \text{if $t > 0$ and $\hat T(t) = 0$} \\
                \frac{u(t) - u(t- \hat T(t))}{\hat T(t)} & \text{otherwise}
            \end{cases}
        \end{align*}
        \end{subequations}
        According to \cite[Lemma~1]{aldsee_tac25} along with $\gop_{N,1}^{1,2L}$ from \Cref{exmp:gop1}, its differentiation error is bounded by
        \begin{align*}
            |[\Diff_{\mathrm{O}} (f + \eta)](t) - \dot f(t)| &\le \begin{cases}
                L t + \frac{2N}{t} & \text{if } t \in (0, \sqrt{2 N/L}) \\
                2 \sqrt{2 N L} & \text{if } t \ge \sqrt{2 N/L}
            \end{cases} \nonumber \\
            &= \gop_{N,1}^{1,2L}(t)
        \end{align*}
        for all $N \ge 0$.
        From \Cref{prop:optimal:exact:equivdef} it is hence seen to be deterministically optimal over $\FL^1$ in the sense of Definition~\ref{def:optimal:exact}.
        Since it is exact from the beginning over $\FL^1$, it also has best worst-case error over $\FL^1$ with respect to noise in $\E_0$ in the sense of Definition~\ref{def:optimal:causal}.
\end{eexample}
\begin{eexample}
        As already mentioned in \Cref{rem:optimalL0}, the first-order linear differentiator $\Diff_\mathrm{A}$ defined as
        \begin{equation*}
        \label{eq:def:asymptotic-diff}
        [\Diff\sbrm{A} u](t) = \begin{cases}
            \frac{u(t) - u(0)}{t} & \text{if $t > 0$} \\
            0 & \text{otherwise},
        \end{cases}
        \end{equation*}
        has best worst-case error over $\F_0^1$ with respect to noise in $\EN$ in the sense of Definition~\ref{def:optimal:causal}, because its differentiation error satisfies
        \begin{equation*}
            |[\Diff_{\mathrm{A}} (f + \eta)](t) - \dot f(t)| \le \frac{2N}{t} = \gop_{N,1}^{1,0}(t)
        \end{equation*}
        for $t > 0$ and $f\in \FL^1, \eta \in \EN$, again using \cite[Lemma 5.4]{seehai_auto23} and $\gop_{N,1}^{1,L}$ with $L = 0$ from \Cref{exmp:gop1}.
        Since this is true for all $N \ge 0$ and $L = 0$, it is also deterministically optimal over $\F_0^1$ in the sense of Definition~\ref{def:optimal:exact}.
\end{eexample}

Up to now it is not clear, however, whether differentiators in the sense of Definitions~\ref{def:optimal:causal} and~\ref{def:optimal:exact} actually exist for arbitrary differentiation order $m > 1$.
The following two subsections hence prove the existence of optimal differentiators with arbitrary order.

\subsection{Existence of Robust Differentiators with Best Worst-Case Error}

In order to fully characterize differentiators with best worst-case error and study their existence, define, for given $L,N \in \RR_{\ge 0}$, and each $\theta \in \RR_{> 0}$ and $i = 1, \ldots, m$ the interval $[\J_{N,i}^{m,L}u](\theta)$ as
\begin{equation}
\label{eq:def:Jti}
    [\J_{N,i}^{m,L}u](\theta) =\left[ \sup_{\substack{h \in \FL^m \\ \|h-u\|_{\infty,\theta} \le N}} h^{(i)}(\theta) - \gop^{m,L}_{N,i}(\theta),\inf_{\substack{h \in \FL^m \\ \|h-u\|_{\infty,\theta} \le N}} h^{(i)}(\theta) + \gop^{m,L}_{N,i}(\theta) \right]
\end{equation}
for $u \in \mathcal{U}$, where it is understood that $[\J_{N,i}^{m,L}u](\theta) = (-\infty, \infty)$ in case the constraints in the supremum and infimum are infeasible.

In the course of this subsection, the following theorem will be proven.
\begin{theorem}
\label{thm:bestworstcase:main}
    Let $m \in \NN$, $L, N \in \RR_{\ge 0}$ and consider the interval $[\J_{N,i}^{m,L}u](\theta)$ defined in \eqref{eq:def:Jti} for $i = 1, \ldots, m$.
    Then,
    \begin{enumerate}[a)]
    \item
    \label{it:bestworstcase:main:existence}
    the interval  
    $[\J_{N,i}^{m,L}u](\theta)$ is nonempty for all $u \in \mathcal{U}$ and all $\theta \in \RR_{> 0}$;
    \item
    \label{it:bestworstcase:main:characterization}
        a causal differentiator $\Diff$ of order $m$ has best worst-case error over $\FL^m$ with respect to noise in $\EN$ if and only if $[\Diff_i u](t) \in [\J_{N,i}^{m,L}u](t)$ for all $u \in \mathcal{U}$ and $t > 0$;
    \item 
    \label{it:bestworstcase:main:construction}
        the differentiator $\Diff$ of order $m$ defined as
        \begin{equation}
        \label{eq:bestworstcase:diff}
            [\Diff_i u](t) = \begin{cases}
                \frac{\inf [\J_{N,i}^{m,L}u](t) + \sup [\J_{N,i}^{m,L}u](t)}{2} & \text{if $t > 0$ and $[\J_{N,i}^{m,L}u](t)$ is bounded} \\
                0 & \text{otherwise}
            \end{cases}
        \end{equation}
        is causal, has best worst-case error over $\FL^m$ with respect to noise in $\EN$, and is robust almost from the beginning over $\FL^m + \E_{N-\delta}$ for all $\delta \in (0,N]$;
        \item 
        \label{it:bestworstcase:main:linear}
        there exist time-varying parameters $d_{t,i,1}, \ldots, d_{t,i,P_{t,i}} \in \RR$, $\tau_{t,i,1}, \ldots, \tau_{t,i,P_{t,i}} \in [0,t]$ with $P_{t,i} \in \NN$, defined for each $t > 0$,
        such that the linear differentiator $\Diff$ of order $m$ defined as
        \begin{equation}
            [\Diff_i u](t) = \begin{cases}
                \sum_{j = 1}^{P_{t,i}} d_{t,i,j} u(t - \tau_{t,i,j}) & \text{if $t > 0$} \\
                0 & \text{otherwise}
            \end{cases}
        \end{equation}
        is causal, has best worst-case error over $\F_{\alpha L}^m$ with respect to noise in $\E_{\alpha N}$ for all $\alpha \in \RR_{\ge 0}$, and is robust almost from the beginning over $\mathcal{U}$ (i.e., strongly robust).
    \end{enumerate}
\end{theorem}
\begin{remark}
    It is worth to point out that the existence of the linear differentiator in item~d)---though also obtainable constructively using other results from literature, such as \cite{bagdas_book98}---here will be proven non-constructively by means of the Hahn-Banach dominated extension theorem.
\end{remark}

\subsubsection{Example}

The following example shows that the differentiator from \cref{thm:bestworstcase:main}, item~\ref{it:bestworstcase:main:construction}), while being robust almost from the beginning over $\FL^m + \E_{N-\delta}$ for all $\delta \in (0, N]$, does not have this property for $\delta = 0$, i.e., over $\FL^m+\EN$.

\begin{example}
Consider the differentiator $\Diff$ defined in \eqref{eq:bestworstcase:diff} for the case of first-order differentiation, i.e., $m = 1$.
For $N,L > 0$, let $T_0 \in \RR_{\ge 0}$, $T_1 = T_0 + \frac{1}{2}\sqrt{\frac{N}{L}}$ and consider the input $u \in \FL^1 + \EN$ defined as
\begin{equation}
    u(t) = \begin{cases}
        N & t \in [0, T_1), t \in \mathbb{Q} \\
        -N & t \in [0, T_1), t \notin \mathbb{Q} \\
        N + \frac{L}{2} (t-T_1)^2 & \text{otherwise}.
    \end{cases}
\end{equation}
Every $h \in \FL^1$ with $\|h-u\|_{\infty} \le N$ necessarily has to be zero on $[0,T_1]$.
Hence, the only possible $h \in \FL^1$ satisfying $\|h-u\|_{\infty} \le N$ is given by
\begin{equation}
    h(t) = \begin{cases}
        0 & t \in [0, T_1) \\
        \frac{L}{2}(t-T_1)^2 & \text{otherwise},
    \end{cases}
\end{equation}
yielding $[\Diff u](t) = \dot h(t)$ for all $t \ge 0$.
For $\delta \in (0, \frac{N}{8})$, consider now a modified input $u_\delta$ defined as
\begin{equation}
    u_{\delta}(t) = \begin{cases}
        -N+\delta & t \in [T_1 - \sqrt{2\delta/L}, T_1), t \notin \mathbb{Q} \\
        u(t) & \text{otherwise}.
    \end{cases}
\end{equation}
Clearly, $\|u_\delta- u\|_{\infty} \le \delta$ for all $t$.
Since all $h_{\delta} \in \FL^1$ with $\|h_{\delta} - u_{\delta}\|_{\infty} \le N$ have to satisfy $h_{\delta}(t) = h(t) = 0$ for $t \in [0, T_1 - \sqrt{2\delta/L})$ and $\ddot h(t) = L$ for $t \ge T_1$, it is easy to see that $\dot h_{\delta}(t) \le \dot h(t) + \sqrt{2 \delta L}$.
Thus, 
\begin{equation}
\label{eq:exmp:sup_h}
\inf [\J_{N,1}^{1,L}u](t) + \gop_{N,1}^{1,L}(t) = \sup_{\substack{h \in \FL^1 \\ \|h-u\|_{\infty,t} \le N}} \dot h(t) \le [\Diff u](t) + \sqrt{2 \delta L}.
\end{equation}

Now, consider one particular hypothesis $h_\delta \in \FL^1$ given by
\begin{equation}
    h_{\delta}(t) = \begin{cases}
        0 & t \in [0, T_1 - \sqrt{2\delta/L}) \\
        \frac{L}{2} (t-T_1+ \sqrt{2\delta/L})^2 & t \in [T_1 - \sqrt{2\delta/L}, T_2) \\
        h(T_2) + 2N - \frac{\delta}{2} + \\
        \qquad [\dot h(T_2) + \sqrt{2 \delta L}](t - T_2) - \frac{L}{2} (t - T_2)^2 & t \in [T_2, T_3]
    \end{cases}
\end{equation}
with $T_{2} = T_1 + \sqrt{2}\frac{N-3\delta/4}{\sqrt{\delta L}}$ and $T_3 = T_2 + \sqrt{\delta/(2L)} + \sqrt{2 N/L}$.
Since $h$ may be written as $h(t) = h(T_2) + \dot h(T_2) (t - T_2) + \frac{L}{2} (t-T_2)^2$ for $t \in [T_2, T_3]$, clearly
\begin{equation}
    h_{\delta}(t) - h(t) = 2N - \frac{\delta}{2} + \sqrt{2 \delta L}(t-T_2) - L (t - T_2)^2
\end{equation}
from which it is straightforward to verify that, indeed, $|h_{\delta}(t) - u_\delta(t)| \le N$ holds for these values of $t$ by using the fact that $u_\delta(t) = u(t) = h(t) + N$.
Furthermore,
\begin{equation}
    \dot h_{\delta}(t) - \dot h(t) = \sqrt{2 \delta L} - 2 L (t - T_2)
\end{equation}
yielding $\dot h_{\delta}(T_3) = \dot h(T_3) - 2 \sqrt{2 N L}$.
Hence,
\begin{equation}
\label{eq:exmp:inf_h}
    \sup [\J_{N,1}^{1,L}u](T_3) - \gop_{N,1}^{1,L}(T_3) = \inf_{\substack{h \in \FL^1 \\ \|h-u\|_{\infty,T_3} \le N}} \dot h(T_3) \le  [\Diff u](T_3) - 2\sqrt{2 N L}.
\end{equation}
Consequently, $[\Diff u_{\delta}](T_3) \le [\Diff u](T_3) - \sqrt{2NL} + \sqrt{\delta L/2}$ is obtained by taking the average of \eqref{eq:exmp:sup_h} at $t = T_3$ and \eqref{eq:exmp:inf_h}, showing that $Q^{\FL^1+\EN}_{\delta,1}(T_3) \ge  \sqrt{2NL}- \sqrt{\delta L/2}$.
Since $T_0$ was arbitrary and $T_3$ increases with $T_0$,
$Q_{\delta,1}^{\FL^1+\EN}(t) > \frac{3}{4} \sqrt{NL}$ holds for all $\delta \in (0,\frac{N}{8})$ and $t \ge 0$, showing lack of any form of robustness of \eqref{eq:bestworstcase:diff} over $\FL^1 + \EN$.
\end{example}

\subsubsection{Technical Results}

In order to prove \cref{thm:bestworstcase:main}, the following functional is introduced.
Let $L, N \in \RR_{\ge 0}$, $m\in\NN$ and define, for each $\theta \in \RR_{> 0}$ and $i = 1, \ldots, m$, the functional $H_{\theta,i} : \FL^m + \EN \to \RR$ as
    \begin{equation}
    \label{eq:def:Hti}
        H_{\theta,i}(u) = \inf_{\substack{h \in \FL^m \\ \|h-u\|_{\infty,\theta} \le N}} h^{(i)}(\theta) + \gop^{m,L}_{N,i}(\theta).
    \end{equation}
\begin{remark}
\label{rem:Hminus}
    It is worth noting that
    \begin{equation}
        -H_{\theta,i}(-u) = \sup_{\substack{h \in \FL^m \\ \|h-u\|_{\infty,\theta} \le N}} h^{(i)}(\theta) - \gop^{m,L}_{N,i}(\theta)
    \end{equation}
    and hence $[\J_{N,i}^{m,L}u](\theta) = [-H_{\theta,i}(-u), H_{\theta,i}(u)]$ whenever $u \in \FL^m + \EN$.
\end{remark}

The functional $H_{\theta,i}$ characterizes differentiators with best worst-case error in the following sense.
\begin{proposition}
\label{prop:optimal:causal:characterization}
Let $L,N\in\RRge$, $m\in\NN$ and $\Diff$ be a causal differentiator of order $m$. Consider $H_{\theta,i}$ as defined in \eqref{eq:def:Hti}. The following statements are equivalent:
\begin{enumerate}[a)]
\item
$\Diff$ has best worst-case error over $\FL^m$ with respect to noise in $\EN$.
\item
\label{it:optimalcausal:interval}
$\Diff$ satisfies $[\Diff_i u](t) \in [-H_{t,i}(-u), H_{t,i}(u) ]$ for all $u \in \FL^m + \EN$, $t > 0$ and $i \in \{1, \ldots, m\}$.
\end{enumerate}
\end{proposition}
\begin{proof}
    To prove that a) implies b), suppose to the contrary, without restriction of generality, that $[\Diff_i u](t) > H_{t,i}(u)$ for some $t$ and $i$.
    From the definition of $H_{t,i}$, there then exists $h \in \FL^m, \eta \in \EN$ such that $u = h + \eta$ and $[\Diff_i u](t) > h^{(i)}(t) + \gop^{m,L}_{N,i}(t)$, contradicting $|[\Diff_i u](t) - h^{(i)}(t)| \le \gop^{m,L}_{N,i}(t)$.

    To show that b) implies a), consider $u = f + \eta$ with $f \in \FL^m$, $\eta \in \EN$.
    Then, $H_{t,i}(u) \le f^{(i)}(t) + \gop^{m,L}_{N,i}(t)$ since $h = f$ is feasible in the definition \eqref{eq:def:Hti}, and hence $-H_{t,i}(-u) \ge f^{(i)}(t) - \gop^{m,L}_{N,i}(t)$.
    Consequently, $|[\Diff_i u](t) - f^{(i)}(t)| \le \gop^{m,L}_{N,i}(t)$ for all $t > 0$, and the claim follows from \cref{def:acc:abs} and the fact that $\gop^{m,L}_{N,i}(t)$ is non-increasing with respect to $t$.
\end{proof}

The following lemma shows some properties of the functional $H_{\theta,i}$ that will be instrumental in proving nonemptiness of the interval $J_{\theta,i}^{L,N}$ and the robustness properties of the differentiators in \cref{thm:bestworstcase:main}.
\begin{lemma}
\label{lem:Hti:properties}
    Let $L,N \in \RR_{\ge 0}$, $\theta \in \RR_{> 0}$, $i,m\in\NN$, $i\le m$, $\delta \in (0,N]$, and consider $H_{\theta,i}$ as defined in \eqref{eq:def:Hti}.
    Then,
    \begin{enumerate}[a)]
        \item 
        \label{it:H:zero}
        $H_{\theta,i}(0) = 0$;
        \item 
        \label{it:H:convexity}
        $H_{\theta,i}$ is convex;
        \item 
        \label{it:H:additivity}
        $H_{\theta,i}(u + f) = H_{\theta,i}(u) + f^{(i)}(\theta)$ for all $f \in \F_0^m$ and all $u \in \FL^m + \EN$;
        \item 
        \label{it:H:interval}
        $|H_{\theta,i}(u) - f^{(i)}(\theta)| \le \gop_{N,i}^{m,L}(\theta)$ whenever $u = f + \eta$ with $f \in \FL^m$, $\eta \in \EN$;
        \item 
        \label{it:H:boundedness}
        $|H_{\theta,i}(u)| \le \gop^{m,L}_{N,i}(\theta)$ whenever $\|u\|_{\infty,\theta} \le N$;
        \item 
        \label{it:H:localLipschitz}
        $|H_{\theta,i}(u_1) - H_{\theta,i}(u_2)| \le 2 \delta^{-1} \gop_{N,i}^{m,L}(\theta)\| u_1 - u_2 \|_{\infty,\theta}$ whenever $u_1,u_2 \in \FL^m+\E_{N-\delta}$;
    \end{enumerate}
\end{lemma}
\begin{proof}
    Item \ref{it:H:zero}) is clear from Remark \ref{rem:Hminus} and the definition of $\gop_{N,i}^{m,L}$ in \eqref{eq:def:G}--\eqref{eq:def:gopt}.
    To show item~\ref{it:H:convexity}), consider arbitrary $u_1, u_2 \in \FL^m+\EN$ and $\lambda \in (0,1)$,
    and note that for every $h_1, h_2 \in \FL^m$, also $h = \lambda h_1 + (1-\lambda) h_2 \in \FL^m$.
Then,
    \begin{align}
        H_{\theta,i}(\lambda u_1 &+ (1-\lambda) u_2) = \inf_{\substack{h \in \FL^m \\ \|h- \lambda u_1 - (1-\lambda) u_2\|_{\infty,\theta} \le N}} h^{(i)}(\theta) + \gop^{m,L}_{N,i}(\theta) \nonumber \\
        &\le \inf_{\substack{h_1,h_2 \in \FL^m \\ \|\lambda (h_1-u_1) + (1-\lambda) (h_2 -u_2) \|_{\infty,\theta} \le N}} \left[\lambda h_1^{(i)} + (1-\lambda) h_2^{(i)}\right] + \gop^{m,L}_{N,i}(\theta) \nonumber \\
        &\le \inf_{\substack{h_1 \in \FL^m \\ \|h_1-u_1 \|_{\infty,\theta} \le N}} \inf_{\substack{h_2 \in \FL^m \\ \|h_2-u_2 \|_{\infty,\theta} \le N}} \left[\lambda h_1^{(i)} + (1-\lambda) h_2^{(i)} \right] + \gop^{m,L}_{N,i}(\theta) \nonumber \\
        &= \lambda H_{\theta,i}(u_1) + (1-\lambda) H_{\theta,i}(u_2),
    \end{align}
    showing convexity of $H_{\theta,i}$.
    For item \ref{it:H:additivity}), consider $f \in \F_0^m$. Then, $h+f \in \FL^m$ holds if and only if $h \in \FL^m$.
    Hence, $H_{\theta,i}(u+f) = H_{\theta,i}(u) + f^{(i)}(\theta)$ follows by substitution.
    Item~\ref{it:H:interval}) is seen from the fact that $H_{\theta,i}(u) \le f^{(i)}(\theta) + \gop^{m,L}_{N,i}(\theta)$ holds, because $f$ is feasible in \eqref{eq:def:Hti}, and $H_{\theta,i}(u) \ge -H_{\theta,i}(-u) \ge f^{(i)}(\theta) - \gop^{m,L}_{N,i}(\theta)$ holds due to convexity and item~\ref{it:H:zero}), using the fact that $0 = \frac{1}{2}u + \frac{1}{2}(-u)$.
    To show item~\ref{it:H:boundedness}), write $u = f+\eta$ with $\eta \in \EN$ and $f(t) = 0$ for $t \in [0, \theta]$.
    The claim is then obtained from item~\ref{it:H:interval}).
For item~\ref{it:H:localLipschitz}), it will be shown that $u = f + \eta$ with $f \in \FL^m$, $\eta \in \E_{N-\delta}$ implies
    \begin{align}
    \label{eq:H:strongduality}
        H_{\theta,i}(u) - \gop_{N,i}^{m,L}(\theta) 
&=\sup_{\lambda \ge 0} \inf_{g \in \FL^m} \left[g^{(i)}(\theta) + \lambda ( \|g - u\|_{\infty,\theta} - N) \right] \nonumber \\
        &= \sup_{\lambda \in [0, \bar \lambda]} \inf_{g \in \FL^m} \left[g^{(i)}(\theta) + \lambda ( \|g - u\|_{\infty,\theta} - N) \right]
    \end{align}
    with $\bar \lambda = 2 \delta^{-1} \gop_{N,i}^{m,L}(\theta)$.
    The first equality therein is obtained from strong Lagrangian duality of \eqref{eq:def:Hti}, because the optimization objective is linear, the constraints are convex, and $f$ is strictly feasible with respect to $\|f -u\|_{\infty,\theta} \le N$.
    To see the second equality in \eqref{eq:H:strongduality}, note that, using \cref{lem:Hti:properties}\ref{it:H:interval}), $f$ satisfies $f^{(i)}(\theta) \le H_{\theta,i}(u) + \gop_{N,i}^{m,L}(\theta)$ and $\|f - u\|_{\infty,\theta} \le N - \delta$.
    Hence $\lambda > \bar \lambda$ implies
    \begin{align}
        \inf_{g \in \FL^m} \left[g^{(i)}(\theta) + \lambda ( \|g - u\|_{\infty,\theta} - N) \right]
        &\le \inf_{\substack{h \in \FL^m \\ \|h - u\|_{\infty,\theta} \le N \\ g = (h+f)/2}} \left[g^{(i)}(\theta) + \lambda ( \|g - u\|_{\infty,\theta} - N) \right] \nonumber \\
        &\le H_{\theta,i}(u) - \frac{\lambda \delta}{2} < H_{\theta,i}(u) - \gop_{N,i}^{m,L}(\theta),
    \end{align}
    contradicting the first equality in \eqref{eq:H:strongduality}
    Consequently, $u_1, u_2 \in \FL^m + \E_{N-\delta}$ satisfy
    \begin{align}
        H_{\theta,i}(u_1) &= \sup_{\lambda \in [0, \bar \lambda]} \inf_{g \in \FL^m} \left[g^{(i)}(\theta) + \lambda ( \|g - u_1\|_{\infty,\theta} - N) \right] + \gop_{N,i}^{m,L}(\theta) \nonumber \\
        &\le \sup_{\lambda \in [0, \bar \lambda]} \inf_{g \in \FL^m} \left[g^{(i)}(\theta) + \lambda ( \|g - u_2\|_{\infty,\theta} - N) \right] + \bar \lambda \|u_1 - u_2\|_{\infty,\theta} + \gop_{N,i}^{m,L}(\theta) \nonumber \\
        &= H_{\theta,i}(u_2) + \bar \lambda \|u_1 - u_2\|_{\infty,\theta}.
    \end{align}
    Analogously, $H_{\theta,i}(u_2) \le H_{\theta,i}(u_1) + \bar \lambda \|u_1 - u_2\|_{\infty,\theta}$ is obtained, proving item~\ref{it:H:localLipschitz}).
\end{proof}

Lipschitz continuity of $H_{t,i}(u)$ on $\FL^m$ with Lipschitz constant $2\gop_{N,i}^{m,L}(t)/N$ according to item~\ref{it:H:localLipschitz}) of Lemma~\ref{lem:Hti:properties} will allow to prove that the differentiator in item~\ref{it:bestworstcase:main:construction}) of \cref{thm:bestworstcase:main} is robust almost from the beginning over $\FL^m + \E_{N-\delta}$ for all $\delta \in (0,N]$.

The following lemma shows existence of a \emph{linear} differentiator with best worst-case error.
In particular, it will be shown that the output of this linear differentiator at any given time instant only depends on the input signal at finitely many time instants.
To that end, define the set
\begin{equation}
\label{eq:set:countablesupport}
    \mathcal{T}_{\theta,i} = \{ \tau \in [0,\theta] : |g(\theta-\tau)| = N \text{ for all  $g \in \mathcal{G}_{L,N}^{m}(\theta)$ with $g^{(i)}(\theta) = \gop^{m,L}_{N,i}(\theta)$} \}
\end{equation}
which, for each fixed $\theta \in \RR_{> 0}$, contains finitely many values due to  \cref{prop:gopt:extremal}\ref{it:extremal:existence}) and is nonempty due to \cref{prop:gopt:extremal}\ref{it:extremal:value}).
Also, recall that a functional $H : \mathcal{U} \to \RR$ is called sublinear if it is subadditive, i.e., $H(u_1+u_2) \le H(u_1)+H(u_2)$ for $u_1,u_2 \in \mathcal{U}$, and positively homogeneous, i.e., $H(\lambda u) = \lambda H(u)$ for $\lambda \in \RR_{> 0}$, $u \in \mathcal{U}$.

\begin{lemma}
    \label{lem:linear:diff}
   Let $L \in \RR_{\ge 0}$, $N, \theta \in \RR_{> 0}$ and consider $H_{\theta,i}$ as defined in \eqref{eq:def:Hti}.
    Then, a sublinear functional $\bar H_{\theta,i} : \mathcal{U} \to \RR$ and a linear functional $D_{\theta,i} : \mathcal{U} \to \RR$ exist satisfying
    \begin{enumerate}[a)]
\item
        \label{it:Dtheta:bound}
        $\bar H_{\theta,i}(u) \le H_{\theta,i}(u)$ for all $u \in \FL^m + \EN$;
        \item
        \label{it:Dtheta:lowerbound}
        $\bar H_{\theta,i}(u) \ge 0$  if $u(\theta-\tau) = 0$ for all $\tau \in \mathcal{T}_{\theta,i}$;
        \item
        \label{it:Dtheta:countablesupport}
        $D_{\theta,i}(u) = 0$  if $u(\theta-\tau) = 0$ for all $\tau \in \mathcal{T}_{\theta,i}$;
        \item 
        \label{it:Dtheta:optimalitybound}
        $D_{\theta,i}(u) \le \bar H_{\theta,i}(u)$ for all $u \in \mathcal{U}$;
        \item 
        \label{it:Dtheta:optimality}
        $D_{\theta,i}(u) \in [-\bar H_{\theta,i}(-u), \bar H_{\theta,i}(u)] \subseteq [ -H_{\theta,i}(-u),  H_{\theta,i}(u)]$ for all $u \in \FL+ \EN$;
        \item
        \label{it:Dtheta:exactness}
        $D_{\theta,i}(f) = f^{(i)}(\theta)$ for all $f \in \F_0^m$; and
        \item
        \label{it:Dtheta:robustness}
        $|D_{\theta,i}(u)| \le \frac{1}{N} \gop^{m,L}_{N,i}(\theta)  \|u\|_{\infty,\theta}$ for all $u \in \mathcal{U}$.
\end{enumerate}
\end{lemma}
\begin{proof}
Define $\bar H_{\theta,i}(u) : \mathcal{U} \to \RR$ as \begin{equation}
\label{eq:def:Hbartheta}
    \bar H_{\theta,i}(u) = \limsup_{\gamma \to 0^+} \frac{1}{\gamma} H_{\theta,i}(\gamma u),
\end{equation}
which is well-defined with $\mathcal{U}$ as its domain because for every $u\in\mathcal{U}$ there exists $\gamma > 0$ sufficiently small so that $\norm[\infty,\theta]{\gamma u} \le N$ and hence $\gamma u$ restricted to the interval $[0,\theta]$ can be extended into a function defined on $\EN$.
Moreover, $\bar H_{\theta,i}$ satisfies
\begin{equation}
    \bar H_{\theta,i}(\lambda u) = \limsup_{\gamma \to 0^+} \frac{1}{\gamma} H_{\theta,i}(\gamma \lambda u) = \limsup_{\gamma\lambda \to 0^+} \frac{\lambda}{\gamma \lambda} H_{\theta,i}(\gamma \lambda u) = \lambda \bar H_{\theta,i}(u)
\end{equation}
for all $\lambda > 0$, showing that it is positively homogeneous.
Since $H_{\theta,i}$ is convex according to \cref{lem:Hti:properties}\ref{it:H:convexity}) and $\limsup$ is subadditive, also $\bar H_{\theta,i}$ is convex, and thus subadditive and sublinear.
Finally, since $H_{\theta,i}(0) = 0$ according to \cref{lem:Hti:properties}\ref{it:H:zero}) and due to its convexity, $H_{\theta,i}(\gamma u) \le \gamma H_{\theta,i}(u)$ holds for all $\gamma \in [0,1]$, $u \in \FL^m+\EN$.
Thus, $\bar H_{\theta,i}(u) \le H_{\theta,i}(u)$ for all $u \in \FL^m+\EN$ follows from \eqref{eq:def:Hbartheta}, showing item~\ref{it:Dtheta:bound}).

For item~\ref{it:Dtheta:lowerbound}) consider any $u \in \mathcal{U}$ satisfying $u(\theta-\tau) = 0$ for all $\tau \in \mathcal{T}_{\theta,i}$.
Then,
\begin{align}
\label{eq:Ht:lowerbound}
    H_{\theta,i}(\gamma u) &= \inf_{\substack{h \in \FL^m \\ \|h-\gamma u\|_{\infty,\theta} \le N}} h^{(i)}(\theta) + \gop^{m,L}_{N,i}(\theta)
    \ge \inf_{\substack{h \in \FL^m \\ |h(\theta-\tau)| \le N \\ \forall \tau \in \mathcal{T}_{\theta,i}}} h^{(i)}(\theta) + \gop^{m,L}_{N,i}(\theta) \nonumber \\
    &= \inf_{\substack{h \in \FL^m \\ \|h\|_{\infty,\theta} \le N}} h^{(i)}(\theta) + \gop^{m,L}_{N,i}(\theta) = H_{\theta,i}(0) = 0
\end{align}
follows for all $\gamma \in \RR_{\ge 0}$.
Therein, the inequality is obtained from the fact that $\|h-\gamma u\|_{\infty,\theta} \le N$ implies $|h(\theta-\tau) - \gamma u(\theta-\tau)| = |h(\theta-\tau)| \le N$ for all $\tau \in \mathcal{T}_{\theta,i}$, while the following equality follows by contradiction using an active constraints argument as follows: 
If there exists $h \in \FL^m$ satisfying $|h(\theta-\tau)| \le N$ for all $\tau \in \mathcal{T}_{\theta,i}$ such that $h^{(i)}(\theta) < - \gop^{m,L}_{N,i}(\theta)$, then there exists a convex combination $\bar h$ of $h$ and all $g \in \mathcal{G}^m_{L,N}(\theta)$ with $g^{(i)}(\theta) = - \gop^{m,L}_{N,i}(\theta)$ (which are exactly those used in the definition of $\mathcal{T}_{\theta,i}$ after sign reversal) that satisfies $\|\bar h\|_{\infty,\theta} \le N$, $\bar h \in \FL^m$, $\bar h^{(i)}(\theta) < -\gop^{m,L}_{N,i}(\theta)$, contradicting $H_{\theta,i}(0) = 0$.
Due to \eqref{eq:Ht:lowerbound} holding for all $\gamma \in \RR_{\ge 0}$ and \eqref{eq:def:Hbartheta}, $\bar H_{\theta,i}(u) \ge 0$ holds.

For item~\ref{it:Dtheta:countablesupport}), note that $\mathcal{U}$ is a linear space and consider the linear subspace $\mathcal{V} \subset \mathcal{U}$ spanned by all functions $v \in \mathcal{U}$ satisfying $v(\theta-\tau) = 1$ for exactly one $\tau \in \mathcal{T}_{\theta,i}$ and $v(t) = 0$ for all other $t \in \RR_{\ge 0}$.
Since $\mathcal{T}_{\theta,i}$ contains only finitely many values, $\mathcal{V}$ is finite dimensional.
Thus, $\mathcal{V}$ has a complement $\mathcal{W}$ such that $\mathcal{V} \oplus \mathcal{W} = \mathcal{U}$, and $u \in \mathcal{W}$ if and only if $u(\theta-\tau) = 0$ for all $\tau \in \mathcal{T}_{\theta,i}$.
Now, partially define the linear functional $D_{\theta,i} : \mathcal{U} \to \RR$ as $D_{\theta,i}(u) = 0$ for all $u \in \mathcal{W}$, making item~\ref{it:Dtheta:countablesupport}) true by construction.
Since $\bar H_{\theta,i}$ is sublinear and $0 = D_{\theta,i}(u) \le \bar H_{\theta,i}(u)$ holds for all $u \in \mathcal{W}$ according to item~\ref{it:Dtheta:lowerbound}), the Hahn-Banach dominated extension theorem (see, e.g. \cite[Section~4.2]{kreysi_book78}) allows to extend $D_{\theta,i}$ to $\mathcal{U}$ such that $D_{\theta,i}(u) \le \bar H_{\theta,i}(u)$ holds for all $u \in \mathcal{U}$, thus fully defining the linear functional $D_{\theta,i}$ and showing also item~\ref{it:Dtheta:optimalitybound}).

Item~\ref{it:Dtheta:optimality}) now follows from linearity of $D_{\theta,i}$ and items~\ref{it:Dtheta:optimalitybound}) and~\ref{it:Dtheta:bound}).
Further, note that $H_{\theta,i}(f) = f^{(i)}(\theta)$ holds for all $f \in \F_0^m$ due to items~\ref{it:H:additivity}) and~\ref{it:H:zero}) of Lemma~\ref{lem:Hti:properties}.
As a consequence, $[-H_{\theta,i}(-f),H_{\theta,i}(f)] = \{ f^{(i)}(\theta) \}$, proving also item~\ref{it:Dtheta:exactness}).

Item~\ref{it:Dtheta:robustness}), finally, follows from
\begin{equation}
        |D_{\theta,i}(u)| \le \frac{\|u\|_{\theta,\infty}}{N} \sup_{\substack{\eta \in \mathcal{U} \\ \|\eta\|_{\infty,\theta} \le N}} |\bar H_{\theta,i}(\eta)| = \frac{\|u\|_{\theta,\infty}}{N} \sup_{\eta \in \EN} |\bar H_{\theta,i}(\eta)| \le \frac{\|u\|_{\theta,\infty}}{N}  \gop^{m,L}_{N,i}(\theta)
\end{equation}
due to linearity of $D_{\theta,i}$ and item~\ref{it:H:boundedness}) of Lemma~\ref{lem:Hti:properties}.
This completes the proof.
\end{proof}

\subsubsection{Proof of the Theorem}

We are now able to prove \cref{thm:bestworstcase:main}.
\begin{proof}[Proof of \cref{thm:bestworstcase:main}]
Recall that $[\J_{N,i}^{m,L}u](\theta) = [-H_{\theta,i}(-u), H_{\theta,i}(u)]$ from \cref{rem:Hminus}.
For item~\ref{it:bestworstcase:main:existence}), note that $H_{\theta,i}(u) + H_{\theta,i}(-u) \ge H_{\theta,i}(0) = 0$ due to items~\ref{it:H:zero}) and~\ref{it:H:convexity}) of \cref{lem:Hti:properties}.
Thus, $H_{\theta,i}(u) \ge - H_{\theta,i}(-u)$, i.e., the interval $[\J_{N,i}^{m,L}u](\theta)$ is non-empty.

Item~\ref{it:bestworstcase:main:characterization}) follows from \cref{prop:optimal:causal:characterization}.

For item~\ref{it:bestworstcase:main:construction}), note first that $\Diff$ is well-defined due to item~\ref{it:bestworstcase:main:existence}). Second, the constraints in \eqref{eq:def:Jti} are feasible and thus $[\J_{N,i}^{m,L}u](\theta)$ is bounded whenever $u \in \FL^m + \EN$. Then, $\Diff$ has best worst-case error over $\FL^m$ with respect to noise in $\EN$ due to item~\ref{it:bestworstcase:main:characterization}).
Causality follows from the fact that $H_{t,i}(u)$ depends only on values of $u$ on the time-interval $[0, t]$ according to its definition \eqref{eq:def:Hti}.
To see robustness almost from the beginning over $\FL^m + \E_{N-\delta}$, consider $u_1 \in \FL^m + \E_{N-\delta}$ and $u_2 = u_1 + \eta$ with $\eta \in \E_{\epsilon}$ with $\epsilon \le \frac{\delta}{2}$.
Hence, $u_1, u_2 \in \FL^m + \E_{N-\frac{\delta}{2}}$.
Applying \cref{lem:Hti:properties}, item~\ref{it:H:localLipschitz}) then yields
\begin{equation}
    |H_{t,i}(u_1) - H_{t,i}(u_2)| \le \frac{4 \epsilon}{\delta} \gop_{N,i}^{m,L}(t),
\end{equation}
which implies $Q_{\epsilon,i}^{\FL^m + \E_{N-\delta}}(t) \le \frac{4 \epsilon}{\delta} \gop_{N,i}^{m,L}(t)$ for $\epsilon \in (0,\frac{\delta}{2}]$, since $\gop_{N,i}^{m,L}(t)$ is non-increasing with respect to $t$.
Taking the limit as $\epsilon \to 0^+$ thus yields $Q_i^{\FL^m + \E_{N-\delta}}(t) = 0$, proving the claimed robustness almost from the beginning over $\FL^m + \E_{N-\delta}$.

To show item~\ref{it:bestworstcase:main:linear}), use the linear functional $D_{\theta,i}$ from \cref{lem:linear:diff} to define the linear differentiator
\begin{equation}
    [\Diff_i(u)](t) = \begin{cases}
        D_{t,i}(u) & \text{if $t > 0$} \\
        0 & \text{otherwise}.
    \end{cases}
\end{equation}
This differentiator has the claimed structure due to item~\ref{it:Dtheta:countablesupport}) of \cref{lem:linear:diff}, which shows that $D_{t,i}$ may be written as
\begin{equation}
    D_{t,i}(u) = \sum_{j = 1}^{P_{t,i}} d_{t,i,j} u(t - \tau_{t,i,j})
\end{equation}
with $\{ \tau_{t,i,1}, \ldots, \tau_{t,i,P_{t,i}} \} = \mathcal{T}_{t,i}$ being the set defined in \eqref{eq:set:countablesupport}, which contains only finitely many values for each $t > 0$ due to \cref{prop:gopt:extremal}.
Causality of the differentiator is clear, because $\mathcal{T}_{t,i} \subseteq [0,t]$ by definition.
Robustness almost from the beginning over $\mathcal{U}$ is obtained because 
\begin{equation}
    |D_{t,i}(u_1) - D_{t,i}(u_2)| = |D_{t,i}(u_1 - u_2)| \le \frac{\|u_1 - u_2\|_{\infty,t}}{N} \gop_{N,i}^{m,L}(t)
\end{equation}
for all $t > 0$ due to linearity,  item~\ref{it:Dtheta:robustness}) of \cref{lem:linear:diff},  and the fact that $\gop_{N,i}^{m,L}(t)$ is non-increasing with respect to $t$.
Best worst-case error over $\FL^m$ with respect to noise in $\EN$ is a consequence of item~\ref{it:bestworstcase:main:characterization}) and \cref{lem:linear:diff}, item~\ref{it:Dtheta:optimality}).
Finally, consider $u = f + \eta$ for some $f \in \F_{\alpha L}^m$, $\eta \in \E_{\alpha N}$ with $\alpha \ge 0$. 
If $\alpha = 0$, the fact that $\Diff$ has best worst-case error over $\F_0^m$ with respect to noise in $\E_0$ follows from item~\ref{it:Dtheta:exactness}) of \cref{lem:linear:diff}. If $\alpha>0$, due to linearity and since $\alpha^{-1} u \in \FL^m + \EN$, then
\begin{equation}
    |[\Diff_i u](t) - f^{(i)}(t)| = \alpha| [\Diff_i \alpha^{-1} u](t) - \alpha^{-1} f^{(i)}(t)| \le \alpha \gop_{N,i}^{m,L}(t) = \gop_{\alpha N,i}^{m,\alpha L}(t)
\end{equation}
is obtained using \eqref{eq:gopt-homogeneity}.
Thus, the linear differentiator $\Diff$ also has best worst-case error over $\F_{\alpha L}^m$ with respect to noise in $\E_{\alpha N}$ for all $\alpha \ge 0$, completing the proof.
\end{proof}

\subsection{Existence of Deterministically Optimal Robust Exact Differentiators}

To characterize and study the existence of deterministically optimal differentiators, define the interval
\begin{align}
\label{eq:def:Iti}
    [\I_i^{m,L}u](\theta) &= \bigcap_{\hat N \in \RR_{\ge 0}} \left( [\J_{\hat N,i}^{m,L} u](\theta) + \big[-\gop^{m,2L}_{\hat N,i}(\theta) + \gop^{m,L}_{\hat N,i}(\theta),\  \gop^{m,2L}_{\hat N,i}(\theta) - \gop^{m,L}_{\hat N,i}(\theta)\big] \right) \nonumber \\
    &= \bigcap_{\hat N \in \RR_{\ge 0}} \left[ \sup_{\substack{h \in \FL^m \\ \|h-u\|_{\infty,\theta} \le \hat N}} h^{(i)}(\theta) - \gop^{m,2L}_{\hat N,i}(\theta),\inf_{\substack{h \in \FL^m \\ \|h-u\|_{\infty,\theta} \le \hat N}} h^{(i)}(\theta) + \gop^{m,2L}_{\hat N,i}(\theta) \right]
\end{align}
for $i,m\in\NN$, $L \in \RR_{\ge 0}$, $i \le m$, $u\in\mathcal{U}$ and $\theta \in \RR_{> 0}$.

In this subsection, we will prove the following theorem.
\begin{theorem}
\label{thm:det-optimal:main}
    Let $m \in \NN$, $L \in \RR_{\ge 0}$ and consider the interval $[\I_i^{m,L}u](\theta)$ defined in \eqref{eq:def:Iti} for $i = 1, \ldots, m$.
    Then,
    \begin{enumerate}[a)]
    \item
    \label{it:det-optimal:main:existence}
    the interval $[\mathcal I_{i}^{m,L} u](\theta)$ is nonempty and bounded for all $u \in \mathcal{U}$ and all $\theta \in \RR_{> 0}$;
    \item
    \label{it:det-optimal:main:characterization}
        a causal differentiator $\Diff$ of order $m$ is deterministically optimal over $\FL^m$ if and only if it satisfies $[\Diff_i u](t) \in [\I_i^{m,L}u](t)$ for all $u \in \mathcal{U}$ and all $t > 0$;
    \item 
    \label{it:det-optimal:main:construction}
        the differentiator $\Diff$ of order $m$ defined as
        \begin{equation}
            [\Diff_i u](t) = \begin{cases}
                \frac{\inf [\I_i^{m,L}u](t) + \sup [\I_i^{m,L}u](t)}{2} & \text{if $t > 0$} \\
                0 & \text{otherwise}
            \end{cases}
        \end{equation}
        is causal, deterministically optimal over $\FL^m$, and robust almost from the beginning over $\mathcal{U}$ (i.e., strongly robust).
    \end{enumerate}
\end{theorem}
\begin{remark}
    Note that the differentiator constructed in item~\ref{it:det-optimal:main:construction}) is \emph{strongly robust} almost from the beginning and thus has a much stronger robustness property than the robustness almost from the beginning over $\FL^m$ featured by all deterministically optimal differentiators according to \cref{prop:optimal:properties}, item~\ref{it:det-optimal:robustness}).
    Indeed, also the deterministically optimal differentiator of order $m = 1$ proposed in \cite[Section~5]{seehai_auto23} (cf. also \cite[Lemma~1]{aldsee_tac25}) is robust, but not strongly robust, almost from the beginning.
\end{remark}

\subsubsection{Technical Results}

In order to eventually show the properties of the interval $[\I_i^{m,L} u](\theta)$ required for proving \cref{thm:det-optimal:main}, define, for each $\theta \in \RR_{> 0}$ and $i = 1, \ldots, m$, the functional $G_{\theta,i} : \RRge \times \mathcal{U} \to \RR \cup \{\pm \infty\}$ as
\begin{equation}
    \label{eq:def:Gti}
        G_{\theta,i}(\hat N, u) = \inf_{\substack{h \in \FL^m \\ \|h-u\|_{\infty,\theta} \le \hat N}} h^{(i)}(\theta) + \gop^{m,2L}_{\hat N,i}(\theta),
    \end{equation}
where $G_{\theta,i}(\hat N, u) = \infty$ if the constraints are infeasible.
\begin{remark}
\label{rem:Iti}
    Note that, for given $L \in \RR_{\ge 0}$ and $m \in \NN$,
    \begin{equation}
        [\I_i^{m,L} u](\theta) = \bigcap_{\hat N \in \RR_{\ge 0}} [-G_{\theta,i}(\hat N, -u), G_{\theta,i}(\hat N, u)]
    \end{equation}
    may be written with the  functional thus defined (cf. also \cref{rem:Hminus}).
\end{remark}

\begin{proposition}
\label{prop:optimal:exact:characterization}
Let $L \in \RR_{\ge 0}$, $m \in \NN$ and consider for each $\theta \in \RR_{> 0}$ the functionals $G_{\theta,i}$ defined in \eqref{eq:def:Gti}.
Then, the following statements are equivalent:
\begin{enumerate}[a)]
\item
$\Diff$ is deterministically optimal over $\FL^m$;
\item
\label{it:optimalexact:interval}
$\Diff$ satisfies $[\Diff_i u](t) \in [-G_{\theta,i}(\hat N, -u), G_{\theta,i}(\hat N, u)]$ for all $\hat N \ge 0$, $t > 0$ and $i \in \{1, \ldots, m\}$.
\end{enumerate}
\end{proposition}
\begin{proof}
To prove that a) implies b), suppose to the contrary that there exist $t$, $i$, and $\hat N$ such that, without restricting generality, $[\Diff_i u](t) > G_{t,i}(\hat N, u)$ holds.
    From the definition of $G_{t,i}$, there then exists $h \in \FL^m, \eta \in \E_{\hat N}$ such that $u = h + \eta$ and $[\Diff_i u](t) > h^{(i)}(t) + \gop^{m,2L}_{N,i}(t)$, contradicting $|[\Diff_i u](t) - h^{(i)}(t)| \le \gop^{m,2L}_{N,i}(t)$ from \cref{prop:optimal:exact:equivdef}.

    To show that b) implies a), consider $u = f + \eta$ with $f \in \FL^m$, $\eta \in \EN$.
    Then, $G_{t,i}(N,u) \le f^{(i)}(t) + \gop^{m,2L}_{N,i}(t)$ since $h = f$ is feasible in the definition \eqref{eq:def:Gti}, and hence $-G_{t,i}(N,-u) \ge f^{(i)}(t) - \gop^{m,2L}_{N,i}(t)$.
    It follows that $|[\Diff_i u](t) - f^{(i)}(t)| \le \gop^{m,2L}_{N,i}(t)$ holds for all $t > 0$, and the claim is obtained from \cref{prop:optimal:exact:equivdef} and the fact that $\gop^{m,2L}_{N,i}(t)$ is non-increasing with respect to $t$.
\end{proof}

The following lemma will be instrumental in showing that the constructed differentiator exists and is strongly robust almost from the beginning. \begin{lemma}
\label{lem:Gti:properties}
    Let $L,N, \epsilon \in \RR_{\ge 0}$, $\theta \in \RR_{> 0}$, $N' \in [N,\infty)$, $m, i \in \NN$, $i \le m$ and consider $G_{\theta,i}$ as defined in \eqref{eq:def:Gti}.
    Then,
    \begin{enumerate}[a)]
    \item 
    \label{it:G:minlength}
    $G_{\theta,i}(N', u) + G_{\theta,i}(N, -u) \ge \gop_{N,i}^{m,2L}(\theta) - \gop_{N,i}^{m,L}(\theta)$ for all $u \in \FL^m + \EN$;
    \item
    \label{it:G:boundedness}
        $|G_{\theta,i}(N,u)| \le 2 \gop^{m,L}_{N,i}(\theta) + \gop^{m,2L}_{N,i}(\theta)$ whenever $\|u\|_{\infty,\theta} \le N$;
\item 
    \label{it:G:robustness}
$ \left|\inf_{\hat N \ge 0} G_{\theta,i}(\hat N,u_1) - \inf_{\hat N \ge 0} G_{\theta,i}(\hat N,u_2)\right| \le   \gop_{\epsilon,i}^{m,2L}(\theta)$ for all $u_1, u_2 \in \mathcal{U}$ satisfying $\|u_1 - u_2\|_{\infty,\theta} \le \epsilon$.
    \end{enumerate}
\end{lemma}
\begin{proof}
    For item~\ref{it:G:minlength}), with the purpose of provoking a contradiction, assume that $G_{\theta,i}( N', u) + G_{\theta,i}( N,-u) < \gop_{N,i}^{m,2L}(\theta) - \gop_{N,i}^{m,L}(\theta)$.
    Then, $h_1, h_2 \in \FL^m$ exist satisfying $\|h_1 - u\|_{\infty,\theta} \le  N'$, $\|h_2 + u\|_{\infty,\theta} \le  N$ such that 
\begin{equation}
        h_1^{(i)}(\theta) + \gop_{ N',i}^{m,2L}(\theta) + h_2^{(i)}(\theta) + \gop_{N,i}^{m,2L}(\theta) < \gop_{N,i}^{m,2L}(\theta) - \gop_{N,i}^{m,L}(\theta),
    \end{equation}
    i.e., $h_1^{(i)}(\theta) + h_2^{(i)}(\theta) < - \gop_{ N',i}^{m,2L}(\theta) - \gop_{N,i}^{m,L}(\theta)$.
    Consider the function $g = h_1 + h_2 \in \F_{2L}^m$.
    This function fulfills 
    $
    \|g\|_{\infty,\theta} \le \|h_1 - u\|_{\infty,\theta} + \|h_2 + u\|_{\infty,\theta} \le  N' +  N
    $
    Applying \eqref{eq:gopt-addnoise} then yields the contradiction
    \begin{equation}
        h_1^{(i)}(\theta) + h_2^{(i)}(\theta) = g^{(i)}(\theta) \ge -\gop_{N' +  N,i}^{m,2L}(\theta) \ge -\gop_{ N',i}^{m,2L}(\theta) - \gop_{N,i}^{m,L}(\theta),
    \end{equation}
    since $g \in \mathcal{G}_{2L,N'+N}^m(\theta)$ holds and it thus satisfies $g^{(i)}(\theta) \ge -\gop_{N' +  N,i}^{m,2L}(\theta)$ by \eqref{eq:def:gopt} and symmetry with respect to sign reversal.
    This completes the proof of item~\ref{it:G:minlength}).

    Item~\ref{it:G:boundedness}) follows from the fact that $G_{\theta,i}(N,u) = H_{\theta,i}(u) - \gop_{N,i}^{m,L}(\theta) + \gop_{N,i}^{m,2L}(\theta)$ and by using item~\ref{it:H:boundedness}) of \cref{lem:Hti:properties}.

    To show item~\ref{it:G:robustness}), consider any $u_1, u_2 \in \mathcal{U}$ satisfying $\|u_1 - u_2\|_{\infty,\theta} \le \epsilon$.
    Then,
    \begin{align}
        G_{\theta,i}(\hat N, u_1) &= \inf_{\substack{h \in \FL^m \\ \|h-u_1\|_{\infty,\theta} \le \hat N}} h^{(i)}(\theta) + \gop^{m,2L}_{\hat N,i}(\theta) 
        \ge \inf_{\substack{h \in \FL^m \\ \|h-u_2\|_{\infty,\theta} \le \hat N+\epsilon}} h^{(i)}(\theta) + \gop^{m,2L}_{\hat N,i}(\theta) \nonumber \\
        &= G_{\theta,i}(\hat N+\epsilon, u_2) - \gop_{\hat N+\epsilon,i}^{m,2L}(\theta) + \gop^{m,2L}_{\hat N,i}(\theta)
    \end{align}
    follows from the fact that $\|h - u_1\|_{\infty,\theta} \le \hat N$ implies $\|h-u_2\|_{\infty,\theta} \le \hat N + \epsilon$.
    Hence,
    \begin{align}
        \inf_{\hat N \in [0, \infty)} G_{\theta,i}(\hat N, u_1)
        &\ge \inf_{\hat N \in [0, \infty)} \left[ G_{\theta,i}(\hat N+\epsilon, u_2) - \gop_{\hat N+\epsilon,i}^{m,2L}(\theta) + \gop^{m,2L}_{\hat N,i}(\theta) \right] \nonumber \\
        &\ge \inf_{\hat N \in [\epsilon, \infty)} G_{\theta,i}(\hat N, u_2) - \sup_{\hat N \in [0, \infty)} \left[ \gop_{\hat N+\epsilon,i}^{m,2L}(\theta) - \gop^{m,2L}_{\hat N,i}(\theta) \right] \nonumber \\
        &\ge \inf_{\hat N \in [0, \infty)} G_{\theta,i}(\hat N, u_2) - \gop_{\epsilon,i}^{m,2L}(\theta),
    \end{align}
    using $\gop_{\hat N + \epsilon,i}^{m,2L}(\theta) \le \gop_{\hat N,i}^{m,2L}(\theta) + \gop_{\epsilon,i}^{m,2L}(\theta)$ according to \eqref{eq:gopt-addnoise} and the fact that $\gop_{\epsilon,i}^{m,L}$ is nondecreasing with respect to $L$.
    Interchanging $u_1$ and $u_2$ furthermore yields also $\inf_{\hat N \ge 0} G_{\theta,i}(\hat N, u_2) \ge \inf_{\hat N \ge 0} G_{\theta,i}(\hat N, u_1) - \gop_{\epsilon,i}^{m,2L}(\theta)$, thus proving the claim.
\end{proof}

For $L,N\in\RRge$, $k,m\in \NN$, $k\le m$, define the function $\psi_{N,k}^{m,L} : \RR_{> 0} \to \RR_{\ge 0}$ as
\begin{equation}
    \psi_{N,k}^{m,L}(\theta) = \begin{cases}
    L \frac{2^{\frac{k}{m+1}}-1}{(m+1-k)!}\left( \frac{1}{\theta} + \frac{\gop_{N,1}^{m,L}(\theta)}{N}  \right)^{- (1 + m - k)} & \text{if $N > 0$} \\
    0 & \text{if $N = 0$}.
    \end{cases}
\end{equation}
The next lemma and the following proposition show how this function may be used to bound the length of the interval $[\mathcal{I}_{k}^{m,L}u](\theta)$ from below.
\begin{lemma}
\label{lem:psi:bound}
    Let $L \in \RR_{\ge 0}$, $T \in \RR_{> 0}$, $m \in \NN$. Let $k\in \{1,\ldots,m\}$.
    Then,
    \begin{equation}
    \label{eq:gopbound:2L:L}
        \gop_{N,k}^{m,2L}(T) - \gop_{N,k}^{m,L}(T) \ge \psi_{N,k}^{m,L}(T)
    \end{equation}
    holds for all $N \in \RR_{\ge 0}$ and $\psi_{N,k}^{m,L}(T)$ is non-decreasing with respect to $N$.
\end{lemma}
\begin{proof}
    The claimed monotonicity of $\psi_{N,k}^{m,L}(T)$  can be seen to follow from the fact that $\psi_{N,k}^{m,L}(T) \ge 0$ and monotonicity of $\gop_{N,1}^{m,L}(T)$ with respect to $L$ by using \eqref{eq:gopt-homogeneity} to write
    \begin{equation}
        \psi_{N,k}^{m,L}(T) = L \frac{2^{\frac{k}{n}}-1}{(n-k)!}\left( \frac{1}{T} + \gop_{1,1}^{m,L/N}(T)  \right)^{k-n}
    \end{equation}
    with $n=m+1$.
    Validity of \eqref{eq:gopbound:2L:L} for $N = 0$ is trivial.
    To prove  the remaining claim, for $N > 0$, define $\tau = [T^{-1} + N^{-1} \gop_{N,1}^{m,L}(T)]^{-1}$.
    It is first shown that $L\tau^n/n! \le N$.
    Otherwise, there would exist $\tilde \tau < \tau$ satisfying $L \tilde \tau^n/n! = N$. The function $f \in \FL^m$ defined via $f(t)=0$ for $t \in [0,T-\tilde\tau)$ and $f(t) = L (t-T+\tilde\tau)^{m+1}/(m+1)!$ for $t\in [T-\tilde\tau,\infty)$ satisfies $f \in \mathcal{G}_{L,N}^m(T)$ defined in \eqref{eq:def:G} and $f^{(1)}(T) = L\tilde\tau^m/m! = nN/\tilde\tau$. This yields the bound $N/\tilde\tau < nN/\tilde \tau \le \gop_{N,1}^{m,L}(T)$ on $\gop_{N,1}^{m,L}(T)$ and thus results in the contradiction $\tau \le N/\gop_{N,1}^{m,L}(T) < \tilde \tau$.
Now, consider $g \in \mathcal{G}_{L,N}^m(T)$ such that $g(T) = N$ and $g^{(k)}(T) = \gop_{N,k}^{m,L}(T)$, whose existence is guaranteed by \cref{prop:gopt:extremal}.
Since, according to \eqref{eq:gopt-timescaling}, $|\dot g(t)| \le \gop_{N,1}^{m,L}(t) \le \frac{T}{t} \gop_{N,1}^{m,L}(T)$ holds for all $t \in (0,T]$, this function is seen to satisfy 
\begin{align}
    g(t) &\ge N - \frac{T}{t} \gop_{N,1}^{m,L}(T) (T-t) \ge N - \frac{T \tau}{T-\tau} \gop_{N,1}^{m,L}(T) 
    = N - \frac{ \gop_{N,1}^{m,L}(T)}{\tau^{-1} - T^{-1}} = 0
\end{align}
for all $t \in [T-\tau,T]$ by using the definition of $\tau$.
Now, consider $h \in \F_{2L}^m$ defined as
\begin{equation}
    h(t) = \begin{cases}
        g(t) & t \in [0, T-\tau) \\
        g(t)-\frac{L(t-T+\tau)^n}{n!} & t \in [T-\tau, T-2^{-\frac{1}{n}} \tau) \\
        g(t)-\frac{L(t-T+\tau)^n}{n!}+\frac{2L (t-T+2^{-\frac{1}{n}} \tau)^n}{n!} & t \in [T-2^{-\frac{1}{n}} \tau, T]
    \end{cases}
\end{equation}
Since $L \tau^n/n! \le N$, it is clear that $h(t) \ge g(t) - N \ge -N$ for $t \in [T-\tau, T]$.
Furthermore, $h(T) = g(T)$, and hence $|h(t)| \le N$ holds for all $t \in [0,T]$ by construction.
Thus,
\begin{align}
    \gop_{N,k}^{m,2L}(T) &\ge h^{(k)}(T) = g^{(k)}(T) - \frac{L \tau^{n-k}}{(n-k)!} + \frac{2L \cdot 2^{-\frac{n-k}{n}} \tau^{n-k}}{(n-k)!} \nonumber \\
    &= \gop_{N,k}^{m,L}(T) + \frac{L (2^{\frac{k}{n}} - 1)}{(n-k)!} \left( \frac{1}{T} + \frac{\gop_{N,1}^{m,L}(T)}{N} \right)^{k-n}
\end{align}
whence the claim follows.
\end{proof}

\begin{proposition}
\label{prop:Iti:length}
    Let $m,k\in\NN$, $k\le m$, $L,\delta \in \RR_{\ge 0}$, $\theta \in \RR_{> 0}$, $u \in \mathcal{U}$.
    Suppose that $\|f - u\|_{\infty,\theta} \ge \delta$ for all $f \in \FL^m$.
    Then, the interval $[\I_k^{m,L} u](\theta)$ is nonempty and has length at least $\psi_{\delta,k}^{m,L}(\theta)$.
\end{proposition}
\begin{proof}
    Note that $G_{\theta,i}(\hat N_1, u) = \infty$ for all $\hat N < \delta$.
    Hence, if the interval $[\mathcal{I}_k^{m,L} u](\theta)$ has length less than $\psi_{\delta,k}^{m,L}(\theta)$ (or if the interval is empty, in case that value is zero), then there exist $\hat N_1, \hat N_2$ with $\hat N_1 \ge \hat N_2 \ge \delta$ such that $G_{\theta,i}(\hat N_1, u) + G_{\theta,i}(\hat N_2, -u) < \psi_{\delta,k}^{m,L}(\theta)$.
    According to \cref{lem:Gti:properties}, item \ref{it:G:minlength}) and \cref{lem:psi:bound}, this yields the contradiction
    \begin{equation}
        G_{\theta,i}(\hat N_1, u) + G_{\theta,i}(\hat N_2, -u) \ge \gop_{\hat N_2,i}^{m,2L}(\theta) - \gop_{\hat N_2,i}^{m,L}(\theta) \ge \psi_{\hat N_2,k}^{m,L}(\theta) \ge \psi_{\delta,k}^{m,L}(\theta),
    \end{equation}
    proving the claim.
\end{proof}

\subsubsection{Proof of the Theorem}

We are now able to prove \cref{thm:det-optimal:main}.
\begin{proof}[Proof of \cref{thm:det-optimal:main}]
    For item~\ref{it:det-optimal:main:existence}), non-emptiness follows from \cref{prop:Iti:length}.
To show also boundedness, consider any $u \in \mathcal{U}$.
    Then, $u$ is uniformly bounded on $[0, \theta]$, i.e., there exists $M$ such that $\|u\|_{\infty,\theta} \le M$.
    According to \cref{lem:Gti:properties}, item~\ref{it:G:boundedness}), $\sup [\I_i^{m,L} u](\theta) \le G_{\theta,i}(M,u) \le 2 \gop_{M,i}^{m,L}(\theta) + \gop_{M,i}^{m,2L}(\theta)$ then holds, with an analogous bound being obtained for the infimum.

    Item~\ref{it:det-optimal:main:characterization}) is obtained by utilizing \cref{prop:optimal:exact:characterization} and noting that $[\Diff_i u](t) \in [\I_i^{m,L} u](t)$ if and only if $[\Diff_i u](t) \in [-G_{t,i}(\hat N, -u), G_{t,i}(\hat N, u)]$ for all $\hat N \ge 0$ (cf. \cref{rem:Iti}).

    To prove item~\ref{it:det-optimal:main:construction}), note that $\Diff$ is well-defined due to item~\ref{it:det-optimal:main:existence}) and deterministically optimal over $\FL^m$ due to item~\ref{it:det-optimal:main:characterization}).
    To show also causality and robustness almost from the beginning over $\mathcal{U}$, write    
    \begin{equation}
        \inf\ [\I_i^{m,L} u](t) = -\inf_{\hat N \in \RR_{\ge 0}} G_{t,i}(\hat N, -u), \qquad
        \sup\ [\I_i^{m,L} u](t) = \inf_{\hat N \in \RR_{\ge 0}} G_{t,i}(\hat N, u).
    \end{equation}
    It obviously suffices to show both properties for $\sup\ [\I_i^{m,L} u](t)$, because they then hold also for the infimum as well as for their average, i.e., for the output of $\Diff$.
    Causality follows from the fact that $\|u_1 - u_2\|_{\infty, \theta} = 0$ implies $\sup\ [\I_i^{m,L} u_1](t) = \sup\ [\I_i^{m,L} u_2](t)$ according to item~\ref{it:G:robustness}) of \cref{lem:Gti:properties}.
    To show robustness almost from the beginning over $\mathcal{U}$, consider $u_1 \in \mathcal{U}$, $u_2 = u_1 + \eta$ with $\eta \in \E_{\epsilon}$.
    As a consequence, also $u_2 \in \mathcal{U}$ and $\|u_1 - u_2\|_{\infty,\theta} \le \epsilon$.
    Again applying \cref{lem:Gti:properties}, item~\ref{it:G:robustness}) thus yields
    \begin{equation}
        \left|\sup\ [\I_i^{m,L} u_1](t) - \sup\ [\I_i^{m,L} u_2](t)\right| \le 
\gop_{\epsilon,i}^{m,2L}(t)
    \end{equation}
    which implies $Q_{\epsilon,i}^{\mathcal{U}}(t) \le  \gop_{\epsilon,i}^{m,2L}(t)$, because the right-hand side is non-increasing with respect to $t$.
    Taking the limit as $\epsilon \to 0^+$ consequently yields $Q_i^{\mathcal{U}}(t) = 0$ for all $t > 0$, completing the proof.
\end{proof}

\section{Conclusions}
\label{sec:conclusions}

This paper has derived the theoretical fundamental limitations of causal differentiators of arbitrary orders in terms of the lowest achievable worst-case differentiation error and has constructively established the existence of deterministically optimal differentiators, which were precisely defined using the results on the achievable worst-case error. Specifically, the constructed deterministically optimal differentiators were shown to be exact from the beginning and strongly robust, the latter being a concept explicitly analyzed here also for the first time. Worthy of mention is the fact that the existence of linear (time-varying) differentiators that achieve the lowest worst-case error under knowledge of a noise amplitude bound was established by application of the Hahn-Banach dominated extension theorem, a technique that seems to not have been applied to differentiators before. Since deterministically optimal differentiators of high orders are time-varying and require knowledge of the whole history of their input, future work may focus on the fundamental theoretical limitations of time-invariant differentiators, and of those requiring only a finite piece of the previous input evolution.

\bibliographystyle{siamplain}
\bibliography{literature}           

\end{document}